\documentclass[leqno,11pt]{amsart}

\usepackage{amssymb}
\usepackage{amsmath}
\usepackage{enumerate}
\usepackage[pdftex]{graphicx}
\usepackage{graphicx, color}
\usepackage[inner=3cm,outer=3cm,top=3cm,bottom=3cm]{geometry}

\usepackage{multirow}

\usepackage{caption}
\usepackage{subcaption}
\usepackage{float}
\usepackage{relsize}
\usepackage{enumitem}
\usepackage{xcolor}
\usepackage{cancel}
\usepackage{stmaryrd}
\usepackage{float}  % Para usar [H]
\usepackage{fancyhdr}  % Para personalizar encabezados
\usepackage{lmodern}
\usepackage{mathrsfs}
\usepackage{mathtools}

\def\s{\mathbb{S}}

\def\R{\mathbb{R}}

\def\C{\mathbb{C}}
\def\L{\mathbb{L}}
\def\E{\mathbb{E}}

\def\K{\mathcal{K}}

\DeclareMathOperator{\arctanh}{arctanh} 

\DeclareMathOperator{\sech}{sech}

\newtheorem{theorem}{Theorem}[section]
\newtheorem{proposition}[theorem]{Proposition}
\newtheorem{corollary}[theorem]{Corollary}

\theoremstyle{definition}
\newtheorem{definition}[theorem]{Definition}
\newtheorem{remark}[theorem]{Remark}
\newtheorem{example}[theorem]{Example}

\numberwithin{equation}{section}

\begin{document}

\title[Quadric surfaces of revolution in the 3-sphere]{Quadric surfaces of revolution in the 3-sphere \\ as cubic Weingarten surfaces}

\author[I. Castro]{Ildefonso Castro}
\address{Departamento de Matem\'{a}ticas \\
	Universidad de Ja\'{e}n \\
	23071 Ja\'{e}n, Spain and IMAG, Instituto de Matem\'aticas de la Universidad de Granada, 18071 Granada, Spain.}
\email{icastro@ujaen.es}

\author[D. L\'opez-L\'opez]{Daniel L\'opez-L\'opez}
\address{Departamento de Matem\'aticas \\
	Universidad de Ja\'{e}n \\
	23071 Ja\'{e}n, Spain.}
\email{dll00018@red.ujaen.es}



\date{}

\begin{abstract}
The study of quadric surfaces of revolution is a cornerstone of classical Euclidean geometry, but its extension to the three-dimensional sphere $\mathbb{S}^3$ has not been sufficiently explored. This article addresses this important gap by providing a rigorous classification and characterization of non-degenerate quadric surfaces of revolution in $\mathbb{S}^3$, namely spherical ellipsoids, hyperboloids and paraboloids, generated by the rotation of spherical conics around a geodesic axis containing their foci or is orthogonal to them.

Using the concept of spherical angular momentum as a prominent geometric invariant, we discover that these surfaces constitute a remarkable class of Weingarten surfaces and prove that they are uniquely characterised by a specific cubic functional relation between their principal curvatures. This result not only provides a unified description of spherical quadric surfaces of revolution, but also highlights a profound geometric universality, reflecting exactly the same cubic Weingarten relations observed in their Euclidean and Lorentzian counterparts.
\end{abstract}

\maketitle

%\tableofcontents

%\newpage

\section{Introduction}\label{Sect:Intro}

The study of quadric surfaces constitutes a classical and fundamental topic in the geometry of Euclidean space $\mathbb E^3$. Among these, quadric surfaces of revolution stand out for their geometric simplicity and symmetry, being naturally generated by the rotation of a conic section around one of its principal axes. This constructive notion, deeply rooted in the Euclidean tradition, admits a natural extension to the three-dimensional sphere $\mathbb{S}^3$. Surprisingly, to the best of our knowledge, this subject has not been addressed in the literature. In this new spherical context, the analysis of such surfaces not only enriches the classical theory but also connects with interesting questions in modern differential geometry. A paradigmatic example of this relevance is the well-known problem posed by S.T.\ Yau  in 1982 (cf.\ \cite{Y82}) regarding a characterization of the ellipsoid of revolution, which inquires whether every closed surface in $\mathbb{E}^3$ satisfying a specific cubic relation between their principal curvatures is necessarily a rotational ellipsoid; see \cite{CC23} and references therein for a detailed explanation of the problem and the different approaches to its solution and generalizations. Simply posing the same questions in $\mathbb{S}^3$ already underscores the necessity of a rigorous and systematic understanding of quadric surfaces of revolution within the spherical framework.

While the generating curves of these surfaces ---the so-called \textit{spherical conics}--- have been the subject of extensive study in classical literature, dating back to the foundational works of Chasles \cite{Cha60} in the 19th century, the theory of the surfaces they generate in $\mathbb{S}^3$ appears surprisingly sparse. Although spherical conics are well-understood as the intersection of the sphere with quadratic cones, or via their focal properties, their role as profile curves for rotational surfaces in $\mathbb{S}^3$ has not received a comparable level of attention. This article addresses this gap, proposing a comprehensive study of quadric surfaces of revolution in the 3-sphere, treating them as rotational surfaces generated by spherical conics rotating around a geodesic axis.
Therefore, the geometric setting for this investigation relies on the description of rotational surfaces in $\mathbb{S}^3$ as objects generated by spherical curves in the spirit of \cite{CCIs24}, \cite{dCD83} or \cite{Ri89} (see Section \ref{Sect:RotS3}). By identifying $\mathbb{S}^3$ with the unit sphere in $\mathbb{R}^4$, we consider a profile curve $\xi$ contained in a totally geodesic 2-sphere of $\s^3$. The rotation of $\xi$ around a fixed great circle $\xi_0$ generates a surface $S_\xi$ in $\mathbb{S}^3$, whose geometry is intrinsically linked to the properties of the generating curve $\xi$. As expected, this approach allows us to translate problems of surface theory in $\mathbb{S}^3$ into the analysis of curves in $\mathbb{S}^2$, providing a powerful reduction of dimension. Recently, this simple idea proved quite fruitful in \cite{CCIs24} in achieving a clear classification of the minimal rotational surfaces of $\mathbb{S}^3$. On the other hand, Brendle showed in \cite{Br13c} that any minimal torus which is immersed in the sense of Alexandrov must be rotationally symmetric.

A crucial ingredient in our analysis is the concept of \textit{spherical angular momentum} (see Section \ref{Sect:Momentum}). As developed in \cite{CCIs23} and \cite{CCIs24}, this geometric invariant plays a decisive role in the classification of spherical curves since, when it is expressed as a function of the non-constant distance from the curve to a fixed great circle $\xi_0$, it captures the relative position of the curve with respect to $\xi_0$ (Theorem \ref{Th:K determines}). Thus it is the natural tool for uniquely determining rotational surfaces with $\xi_0$  as an axis of revolution (Corollary \ref{cor:Kkey}). Consequently, the geometry of the rotational surface $S_\xi$ ---including its principal curvatures--- can be completely encoded in terms of the spherical angular momentum (with respect to $\xi_0$) of its profile curve $\xi$ (Corollary \ref{cor:edoW}). 

The main objective of this paper is to formally introduce and characterize \textit{non-degenerate quadric surfaces of revolution} in $\mathbb{S}^3$. We define these surfaces in Definition \ref{def:quadric} by rotating non-degenerate spherical conics—specifically, spherical ellipses, hyperbolas, and parabolas—around a geodesic containing their foci or orthogonal to them.  We derive explicit parametrizations and implicit equations for the resulting surfaces, consistently identifying families of spherical (prolate and oblate) ellipsoids, hyperboloids, and paraboloids of revolution. The comprehensive systematic study carried out on spherical conics in Section \ref{Sect:Conics}, not only as geometric loci (Section \ref{Sect:ConicsLoci}) but also,  based on this, as the intersection of the sphere with elliptical cylinders (Section \ref{Sect:ConicsIntersection}), allows us to characterize them (and the quadric surfaces they generate) in terms of their corresponding spherical angular momentum.

As an application, the central result of this paper is the discovery that all such non-degenerate quadric surfaces of revolution in $\mathbb{S}^3$ belong to an interesting class of Weingarten surfaces. There do not appear to be many results on Weingarten surfaces in $\s^3$ (see e.g.\ \cite{Br14}, \cite{LYZ22}). Specifically, we prove in Theorem \ref{Th:Main} that ---apart from those that can be considered degenerate quadric surfaces of revolution--- the non-degenerate ones are the only ones satisfying the remarkable cubic Weingarten relation $k_m = \mu \, k_p^3$ between their principal curvatures $k_m$ (along meridians) and $k_p$ (along parallels), where $\mu$ is a real constant. 
This relation not only provides a unified characterization for the entire family of quadric surfaces of revolution but also mirrors the behavior observed in \cite{CC22} and \cite{CCC26} in their Euclidean and Lorentzian counterparts, revealing a deep universality in the geometry of rotational quadric surfaces across different space forms. 

Finally, we would like to highlight that stereographic projections on $\E^3$ of non-degenerate (resp.\ degenerate) quadric surfaces of revolution in $\s^3$ are Darboux (resp.\ Dupin) cyclides, see Remark \ref{rm:Darboux}. The geometric modeling community is interested in using Darboux cyclides' circular arc structure in modern free form architecture (see e.g.\ \cite{P12}).

\section{Rotational surfaces in the 3-sphere generated by spherical curves} \label{Sect:RotS3}

\subsection{Spherical curves}\label{Sect:CurvesS2}

Let $\s^2 = \{ (x,y,z)\in \R^3 : x^2+y^2+z^2=1 \}$ be the unit 2-sphere in $\R^3$ and
consider  $\xi=(x,y,z): I\subseteq \R \rightarrow \s^2 $ a spherical smooth curve parameterized by the arc length,
i.e.\ $|\xi (s)|=|\dot \xi (s)|=1, \, \forall s\in I$, where $I$ is some interval in $\R$. 
We denote by a dot $\dot \,$ the derivative with respect to $s$ and by $\langle
\cdot, \cdot \rangle $ and $\times$ the Euclidean inner product and the cross product in $\R^3$ respectively. 

Let $T=\dot \xi $ be the unit tangent vector and $N=\xi \times \dot
\xi$ the unit normal vector of $\xi$. If $\nabla$ is the connection in $\s^2$, the oriented geodesic curvature $\kappa$ of $\xi$
is given by the Frenet equation $\nabla_T T = \kappa N$. Hence, we have that
$\ddot \xi =-\xi + \kappa N$, $ \dot N=-\kappa \,  \dot \xi$,
and so $\kappa = \det (\xi, \dot \xi, \ddot \xi)$.

If we use geographical coordinates $(\lambda, \varphi)$ on $\s^2$, $-\pi < \lambda \leq \pi$, $-\pi/2 \leq \varphi \leq \pi/2$, we can write
$\xi(s)=(\cos \varphi (s)  \cos \lambda (s) , \cos \varphi (s)  \sin \lambda (s) , \sin \varphi (s))$, $s \in I$.

\begin{example}\label{ex:parallel}
	Given $\varphi_0 \in (-\pi/2,\pi/2)$, let $\xi_{\varphi_0} $ be the \textit{parallel} of latitude $\varphi_0$ in $\s^2$  given by $z=\sin \varphi_0$. Its arc length parametrization is given by 
	 $\xi_{\varphi_0}(s) = \left(c_{\varphi_0}\cos \left(s/c_{\varphi_0}\right), c_{\varphi_0}\sin \left(s/c_{\varphi_0}\right), s_{\varphi_0}\right)$,  $s\in\mathbb{R}$,  where  $c_{\varphi_0}$ and $s_{\varphi_0}$ denote $\cos \varphi_0$ and $\sin \varphi_0$, respectively.
	Its constant curvature is $\kappa = \tan \varphi_0$. In particular, the great circle $\xi_0$ will be called the \textit{equator} of $\s^2$.
\end{example}

\begin{example}\label{ex:great circle}
	Given $\theta \in (0,\pi) $, let $\mu_\theta $ be the \textit{great circle} in $\s^2$ given by $\sin \theta \, y = \cos \theta \, z$, $\theta $ being the angle between $\xi_0$ and $\mu_\theta$. Putting $c_\theta=\cos \theta$ and $s_\theta=\sin \theta$, its arc length parametrization is given by
	$\mu_\theta (s)=(\cos s, c_\theta \sin s, s_\theta \sin s) $, $s \in \R$.
	In particular, the great circle $\mu_{\pi/2} $ orthogonal to $\xi_0$ will be called the \textit{zero-meridian} of $\s^2$.
\end{example}

\begin{example}\label{ex:small circle}
	Given $\delta \neq 0$, let $\eta_{\delta}$ be the \textit{small circle}  in $\s^2$ orthogonal to $\xi_0$ given by $y=\tanh \delta$. Abbreviating $ch_\delta = \cosh \delta$ and $th_\delta = \tanh \delta$, its arc length parametrization is given by
	$\eta_\delta (s)=\left(\cos (ch_\delta s)/ch_\delta, th_\delta,\sin (ch_\delta s)/ch_\delta\right)$, $s\in \R$, and its
	 constant curvature is $\kappa = -\sinh \delta$. If $\delta =0$, we recover the zero-meridian $\mu_{\pi/2}$.
\end{example}

\subsection{Rotational surfaces in $\s^3$}\label{Sect:SurfacesS3}

Throughout this paper we will identify the 3-sphere $\s^3$ with the unit sphere in $\R^4$; that is,
$\s^3 = \{ (x_1,x_2,x_3,x_4)\in \R^4 : x_1^2+x_2^2+x_3^2+x_4^2=1 \}$.
To study rotational surfaces in $\s^3$, we will consider the {\em generating curve} $\xi=(x,y,z)$ in the totally geodesic 2-sphere $x_4=0$, and the equator $\xi_0$ as the {\em axis of revolution}. We will denote by $S_\xi$ the corresponding surface and we deduce its following parametrization:
\begin{equation}\label{eq:paramXrot}
S_\xi \equiv X(s,t) = 
\begin{pmatrix}
	1 & 0 & 0 & 0\\
	0 & 1 & 0 & 0\\
	0 & 0 & \cos t & -\sin t\\
	0 & 0 & \sin t & \cos t
\end{pmatrix}
\begin{pmatrix}
	x(s)\\
	y(s)\\
	z(s)\\
	0
\end{pmatrix}
=(x(s),y(s),z(s)\cos t, z(s)\sin t),
\end{equation}
$s\in I \subseteq \mathbb{R}$, $t\in (-\pi,\pi)$. If one can write $\xi \equiv z=R(x,y)\neq 0$, then the implicit equation of $S_\xi $ is given by $x_3^2+x_4^2=z^2=R(x_1,x_2)^2$.
It is a straightforward computation to control the local geometry of $S_{\xi}$ through its first and second fundamental forms with respect to $(s,t)$-coordinates, obtaining $I\equiv ds^2+z^2dt^2$ and $II\equiv \kappa \, ds^2 + z (\dot x y - x \dot y)dt^2$. Therefore, the principal curvatures $k_{\text m}$ and $k_{\text p}$ are reached along the meridians $dt=0$ and the parallels $ds=0$ of $S_{\xi}$ respectively, and are given by
\begin{equation}\label{eq:kmkp_1}
	k_{\text m}=\kappa, \quad k_{\text p}=\frac{\dot x y - x \dot y}{z}.
\end{equation}
We point out that if we use geographical coordinates for $\xi=(\lambda,\varphi)$ and look at the 3-sphere as $\s^3=\{ (\omega_1,\omega_2)\in \C^2 : |\omega_1|^2+|\omega_2|^2=1 \}$, the previous parametrization 
$$
X(s,t)=\left( 
x(s) , y(s) , 
z(s) \cos t , z(s) \sin t \right)
$$
of $S_\xi$ is rewriting as
\begin{equation}\label{eq:paramXgeo}
X(s,t)=\left( 
\cos \varphi (s) \,e^{i\lambda (s)} , \sin \varphi (s)\, e^{it} \right).
\end{equation}
\begin{remark}\label{rem:axially}
The rotational surfaces $S_\xi$ given by \eqref{eq:paramXrot} or \eqref{eq:paramXgeo} are congruent to the {\em axially symmetric} (with respect to the geodesic $\xi_0$) surfaces in $\s^3$ considered in \cite{AL15}, \cite{Br13b}, \cite{Pr10} or \cite{Pr16}. However, they choose the plane curve $\alpha(s)=(x(s),y(s))$ contained in the unit disk as profile curve, and then $z(s)=\sqrt{1-|\alpha(s)|^2}$.
\end{remark}

\begin{example}\label{ex:tot_geod_S3}
The {\em totally geodesic equatorial sphere} $\s^2 \hookrightarrow \s^3$, given by $x_2=0$, can be seen as the rotational surface generated by the zero-meridian  $\mu_{\pi/2} $ (see Example \ref{ex:great circle}). Note that $S_{\mu_{\pi/2}}\equiv x_3^2+x_4^2=1-x_1^2$. Obviously, $k_{\text m}=0=	k_{\text p}$.
\end{example}

\begin{example}\label{ex:tot_umb_S3}
	The {\em totally umbilical spheres} 
	$\s^2(R_\delta) \hookrightarrow\s^3$, $0<R_\delta = \sech \delta \leq 1$, given by $x_2=\tanh \delta$, are the rotational surfaces generated by the small circles $\eta_{\delta} $, $\delta \geq 0$ (see Example \ref{ex:small circle}). Note that $S_{\eta_{\delta}}\equiv x_3^2+x_4^2=R_\delta^2-x_1^2$.
	If $\delta =0$, we recover the totally geodesic equatorial 2-sphere $x_2=0$ (Example \ref{ex:tot_geod_S3}). See Figure \ref{fig:TotUmbS2}.
	Using \eqref{eq:kmkp_1}, we get that $k_{\text m}=-\sinh \delta=	k_{\text p}$.
\end{example}

\begin{figure}[h!]
	\begin{center}
		\includegraphics[height=5.2cm]{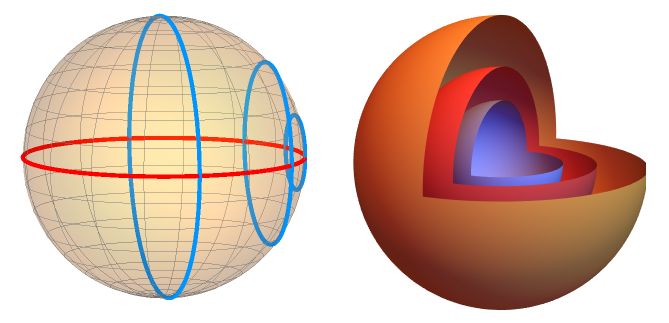}
		\caption{ 
			Generating curves $\eta_{\delta} $, $\delta \geq 0$, and open sights of a stereographic projection of totally umbilical spheres
		$S_{\eta_{\delta}}$.
		}
		\label{fig:TotUmbS2}
	\end{center}
\end{figure}

\begin{example}\label{ex:standard_tori}
	The {\em standard tori} $\s^1(\cos\varphi_0)\times \s^1\left( \sin \varphi_0 \right) \hookrightarrow\s^3$  are the rotational surfaces generated by the parallels $\xi_{\varphi_0}$, $\varphi_0 \in (0,\pi/2)$ (see Example \ref{ex:parallel}). Note that $S_{\xi_{\varphi_0}}\equiv x_3^2+x_4^2=\sin^2 \varphi_0$. See Figure \ref{fig:StandardTori}. Using \eqref{eq:kmkp_1}, their principal curvatures are $k_{\text m}=\tan \varphi_0 $ and $ k_{\text p}=-\cot \varphi_0 $. Therefore they are CMC flat tori. The {\em Clifford torus}, corresponding to $\varphi_0=\pi/4$, is the only embedded minimal (i.e.\ with zero mean curvature) surface in $\s^3$ of genus one (cf.\ \cite{Br13a}). 
\end{example}

\begin{figure}[h!]
	\begin{center}
		\includegraphics[height=5.2cm]{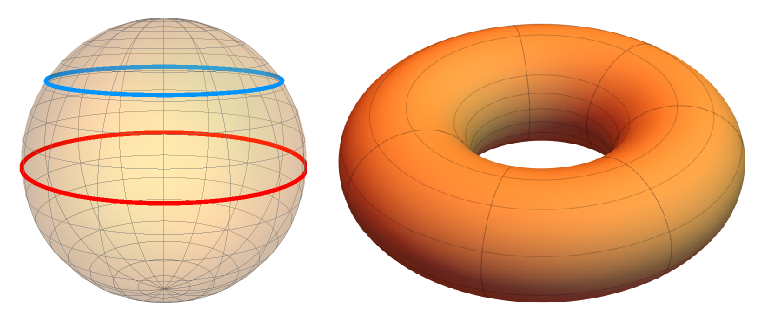}
		\caption{ 
			Generating curve $\xi_{\varphi_0}$, $\varphi_0 \in (0,\pi/2)$, and stereographic projection of a standard torus
			$S_{\xi_{\varphi_0}}$.
		}
		\label{fig:StandardTori}
	\end{center}
\end{figure}

\begin{example}\label{ex:lun_esf_S3}
	The rotational surfaces generated by the great circles $\mu_\theta $, $\theta \in (0,\pi) $, $\theta \neq \pi/2$ (see Example \ref{ex:great circle}) will be called {\em spherical moons}. See Figure \ref{fig:SphMoon}. Observe that 
	$S_{\mu_{\theta}}\equiv x_3^2+x_4^2=\tan^2 \theta \, x_2^2$ and so they can be thought as spherical cones. Clearly, using \eqref{eq:kmkp_1}, $k_{\text m}=0$. 
	\end{example}

\begin{figure}[h!]
	\begin{center}
		\includegraphics[height=5.2cm]{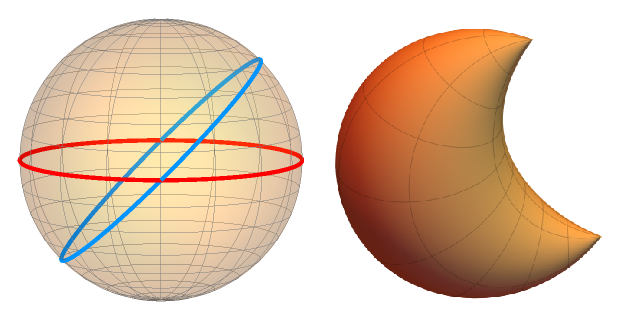}
		\caption{ 
			Generating curve $\mu_\theta $, $\theta \in (0,\pi) $, $\theta \neq \pi/2$, and stereographic projection of a spherical moon
			$S_{\mu_{\theta}}$.
		}
		\label{fig:SphMoon}
	\end{center}
\end{figure}

%\vspace{0.2cm}

%%%%%%%%%%%%%%%%%%%%%%%%%%%%%%%%%%%%%%%%%%%%%%%%%%%%%%%%%%%%%%%%%%%%%%%%%%%%%%%%%%%%%%%

\section{The spherical angular momentum} \label{Sect:Momentum}

\subsection{The spherical angular momentum of a spherical curve}\label{Sect:KCurvesS2}

Following \cite{CCIs23} or \cite{CCIs24}, we introduce a smooth function $\mathcal K$ associated to any spherical curve $\xi $ (see Section \ref{Sect:CurvesS2}), which completely determines it (up to a family of distinguished isometries) in relation with its relative position with respect to the fixed geodesic $\xi_0$. 

At any given point $\xi (s)=(x(s),y(s),z(s)), \, s\in I$, we introduce the {\em spherical angular momentum} (with respect to the equator $\xi_0$), denoted by $\mathcal K (s)$, as the vertical component of the unit normal vector $-N(s)$. Concretely, we define
\begin{equation}\label{eq:spherical momentum}
	\mathcal K(s) := -\langle N(s), (0,0,1) \rangle = \dot x(s) y(s) -x(s) \dot y(s),
\end{equation}
where $s$ is the arc-length parameter of $\xi$.
In physical terms,  $\mathcal K (s)$ may be described as the angular momentum of a particle of unit mass with unit speed and spherical trajectory $\xi (s)$. We point out that $\mathcal K$ is a smooth function that takes values in $[-1,1]$. It is well defined, up to sign, depending on the orientation of the normal to $\xi$. And it easy to check that it is invariant under rotations around $z$-axis of $\xi$.

\begin{example}\label{rm:Kcircles}
The spherical angular momentum of the great circles $\mu_\theta$, $\theta \in (0,\pi)$, described in Example \ref{ex:great circle}, is constant, concretely $\mathcal K_\theta=-\cos \theta$. So it distinguishes all the geodesics in $\s^2$ by its relative position with respect to the equator $\xi_0$. In particular, the zero-meridian $\mu_{\pi/2}$ has null spherical angular momentum.
%It is easy to check that the spherical angular momentum of the parallels $\xi_{\varphi_0}$, $\varphi_0\in (0,\pi/2)$, is also constant $-\cos \varphi_0$ (see Example \ref{ex:parallel}). 
For the small circles $\eta_\delta$, $\delta \neq 0$, described in Example \ref{ex:small circle}, we obtain that $\mathcal K_\delta (s)=- th_\delta  \sin(ch_\delta  s)=-\sinh \delta \, z_\delta(s)$,  being $\kappa_\delta=-\sinh \delta$ precisely the curvature of $\eta_\delta$.
\end{example}

We now recall the main local result of this section from \cite[Theorem 3.1]{CCIs24}, which 
shows how the spherical angular momentum $\mathcal K = \mathcal K (z) $ determines uniquely the spherical curve $\xi =(x,y,z)$, assuming $z$ non-constant. To this end, 
we pay attention to the geometric condition that the curvature $\kappa$ of $\xi$ depends on the distance from $\xi=(x,y,z)$  to $\xi_0\equiv z=0$.
So we can assume, at least locally, that $\kappa=\kappa(z)$.

\begin{theorem}\label{Th:K determines}
	Any spherical curve $\xi=(x,y,z):I\subseteq \R \rightarrow \s^2$, with $z$ non-constant, is uniquely determined by its spherical angular momentum $\mathcal K$ as a function of its distance to the equator $\xi_0$, that is, by $\mathcal K= \mathcal K(z)$. The uniqueness is modulo rotations around the $z$-axis. Moreover, the curvature of $\xi $ is given by $\kappa (z)=\mathcal K' (z)$.
\end{theorem}

\begin{remark}\label{rm:algorithm}
 Along the proof of Theorem \ref{Th:K determines} in \cite[Theorem 3.1]{CCIs24},
		we even determine by quadratures in a constructive explicit way the spherical curve $$\xi=(x,y,z)=(\cos \varphi  \cos \lambda , \cos \varphi  \sin \lambda , \sin \varphi )$$ determined by $\mathcal K= \mathcal K(z)$. Concretely:
		\begin{enumerate}
			\item The arc-length parameter $s$  in terms of $z$, is given ---up to translations of the parameter--- by the integral:
			$$
			s=s(z)=\int\!\frac{dz}{\sqrt{1-z^2-\mathcal K(z)^2}},
			$$
			where $\mathcal K(z)^2 + z^2 < 1 $, and we
			invert $s=s(z)$ to get $z=z(s)$ and the latitude 
			$
			\varphi(s)=\arcsin z(s)
			$ of $\xi$.
			\item The longitude of $\xi$ in terms of $s$, is given  ---up to a rotation around the $z$-axis--- by the integral:
			$$
			\lambda (s)=\int \! \frac{\mathcal K(z(s))}{z(s)^2-1}ds ,
			$$
			where $|z(s)|<1$.
		\end{enumerate}
\end{remark}

We show a simple example of the algorithm given in  Remark~\ref{rm:algorithm}:

\begin{example}[$\kappa\!\equiv\! 0$]\label{ex:Kcte}
	{\rm Then $\mathcal K \!\equiv \! c \!\in\! \R$, $s\!=\arcsin \frac{z}{\sqrt{1-c^2}} $, with $|c|<1$.
		So $z(s)=\sqrt{1-c^2} \sin s $, $\lambda(s)\!=\!-\!\arctan (c \tan s)$ and, finally, $\xi (s)\!=\!(\cos s,-c \sin s, \sqrt{1-c^2} \sin s) $, which corresponds to the great circle $\sqrt{1-c^2}\,y+c\,z=0 $.
		Up to rotations around the $z$-axis, they provide arbitrary great circles in $\s^2$, except the equator. Putting $c=-\cos \theta$, we recover the $\mu_\theta$ given in Example \ref{ex:great circle}. As a consequence of Theorem~\ref{Th:K determines}, the great circle $\mu_\theta $ is the only spherical curve (up to rotations around the $z$-axis) with constant spherical angular momentum $\mathcal K_\theta\!\equiv\! -\cos \theta$. % (see Figure~\ref{Circles}). 
	}
\end{example}

\subsection{The spherical angular momentum of a rotational surface in $\s^3$}\label{Sect:KRotSurfacesS3}

 It is clear that the geometry of a rotational surface $S_\xi$ is completely controlled by the geometry of its generating curve $\xi$ (see Section \ref{Sect:SurfacesS3} for details). As a consequence of Theorem \ref{Th:K determines} and the fact that rotations around the $z$-axis in $\s^2 $ are nothing but translations along $\xi_0$ in $\s^3$, we conclude immediately the following interesting result.
\begin{corollary}\label{cor:Kkey}
	Any rotational surface $S_\xi $  in $\s^3$, where $\xi=(x,y,z)$ with $z$ non-constant, is uniquely determined, up to translations along $\xi_0$, by the spherical angular momentum $\mathcal K=\mathcal K(z)$ of its generating curve $\xi $.
\end{corollary}

\begin{remark}\label{rm:Ksimple}
	Corollary \ref{cor:Kkey} means that, apart from the standard tori corresponding to $z$ constant (see Example \ref{ex:standard_tori}), any rotational surface in the 3-sphere is determined by its spherical angular momentum $\mathcal K=\mathcal K(z)$, i.e.\ the one of its generating curve. 
	For instance, the totally geodesic equatorial sphere, see Example \ref{ex:tot_geod_S3}, is uniquely determined by a null spherical angular momentum $\mathcal K\equiv 0$. In addition, the spherical moons in Example \ref{ex:lun_esf_S3} are uniquely determined by a non-null constant spherical angular momentum (see Example \ref{rm:Kcircles}).
	Moreover, the totally umbilical sphere of radius $R_\delta = \sech \delta \leq 1$ in Example \ref{ex:tot_umb_S3} is uniquely determined by a linear spherical angular momentum $\mathcal K(z)=k_0 z$, being $k_0=-\sinh \delta$, $\delta \geq 0$ (see Example \ref{rm:Kcircles}).
\end{remark}

On the other hand, \textit{Weingarten surfaces} are defined by a functional relation between their principal curvatures, see for instance \cite{Ch45} or \cite{W61}.
For the particular case of rotational Weingarten surfaces, we simply write $\Phi(k_{\text m}, k_{\text p})=0$, with $\Phi $ a differentiable function.
Thus, taking into account \eqref{eq:kmkp_1}, \eqref{eq:spherical momentum} and Theorem \ref{Th:K determines}, we express the principal curvatures of a rotational surface $S_\xi $  in $\s^3$, with $\xi=(x,y,z)$, in terms of the spherical angular momentum $\mathcal K=\mathcal K(z)$ of its generating curve $\xi $ and the non-constant distance $z$ from $\xi$ to the equator $\xi_0$:
\begin{equation}\label{eq:kmkp}
	k_{\text m}=\mathcal K'(z), \quad k_{\text p}=\frac{\mathcal K(z)}{z}.
\end{equation}
Therefore, an immediate consequence of \eqref{eq:kmkp} is the following key result.
\begin{corollary}\label{cor:edoW}
	Any Weingarten relation $\Phi(k_{\text m}, k_{\text p})=0$ on a rotational surface $S_{\xi} $, $\xi=(x,y,z)$, in  $\s^3$ translates through \eqref{eq:kmkp} into a first-order ordinary differential equation $$\hat \Phi (z,\mathcal K , \mathcal K')=0$$ for the  spherical angular momentum $\mathcal K=\mathcal K (z)$  of its generatrix curve as a function of the non-constant distance $z$ from $\xi$ to the equator $\xi_0$.
\end{corollary}	
We can illustrate it with three simple examples:

Suppose $k_{\rm p}=0$. Then \eqref{eq:kmkp} gives $\mathcal K\equiv 0$. Using Remark \ref{rm:Ksimple}, we conclude that {\em the totally geodesic equatorial sphere is the only rotational surface in $\s^3$ with $k_{\rm p}=0$}.
	
Suppose $k_{\rm m}=0$. Now \eqref{eq:kmkp} implies that $\mathcal K$ is constant. Using Remark \ref{rm:Ksimple} again, we deduce that {\em the totally geodesic equatorial sphere and the spherical moons are the only rotational surfaces in $\s^3$ with $k_{\rm m}=0$}.
	
Suppose $k_{\rm m}=k_{\rm p}$. From \eqref{eq:kmkp} we get $\mathcal K' = \mathcal K /z$, whose general solution is $\mathcal K(z)=k_0 z$, $k_0 \in \R$. In this way, using Remark \ref{rm:Ksimple}, we corroborate that the rotational 2-spheres described in Example \ref{ex:tot_umb_S3} are the only totally umbilical.
	
If one imposes $k_{\rm m}=-k_{\rm p}$, one is heading towards the classification of rotational minimal surfaces in $\s^3 $ given in \cite[Theorem 5.2]{CCIs24} in which the totally geodesic equatorial sphere, the Clifford torus (see Example \ref{ex:standard_tori}) and the spherical catenoids (see \cite[Section 3.2]{CCIs24}) appear. 
	
%\vspace{0.2cm}

%%%%%%%%%%%%%%%%%%%%%%%%%%%%%%%%%%%%%%%%%%%%%%%%%%%%%%%%%%%%%%%%%%%%%%%%%%%%%%%%%%%%%%%

\section{Spherical conics} \label{Sect:Conics}
In this section, we study spherical conics from two different perspectives: one more geometric and the other more analytical. For the first overview, we rely particularly  on \cite{T06}. 

\subsection{Spherical conics as geometric loci} \label{Sect:ConicsLoci}

Considering the intrinsic geodesic distance $dist$ on $\s^2$, we can define the spherical conics as geometric loci in the usual way:  Given two fixed different points $F_1$ and $F_2$ on the sphere $\s^2$, 
the set 
$$
	 \mathcal E _{d,e}(F_1,F_2)=\{ P\in \s^2 :  dist(P,F_1)+dist(P,F_2)=2d , \  \pi/2>d>e=dist(F_1,F_2)/2>0 \} 
$$
is called a \textit{spherical ellipse} (of parameters $d,e$) with foci $F_1$ and $F_2$, and the set 
$$ \mathcal H _{d,e}(F_1,F_2) =\{ P\in \s^2 : | dist(P,F_1)-dist(P,F_2) |=2d, \  0<d<e=dist(F_1,F_2)/2<\pi/2 \} $$
is called a \textit{spherical hyperbola} (of parameters $d,e$) with foci $F_1$ and $F_2$. 
While $2e$ is the focal distance, $2d$ can be interpreted as the distance between the vertices, i.e.\ the points of the conic on the great circle connecting the foci. It is obvious that a spherical hyperbola has two branches: 
$$ 
\mathcal H^1_{d,e}(F_1,F_2) \equiv dist(P,F_1)-dist(P,F_2) =2d, \quad
\mathcal H^2_{d,e}(F_1,F_2) \equiv dist(P,F_2)-dist(P,F_1) =2d. 
$$

Given $P\in \s^2$, we denote $\overline P:=-P$ its antipodal point. In the same way, $\overline S$ stands for the image of $S\subset \s^2$ under the antipodal map. If we also abbreviate $\widehat{PQ}=dist(P,Q)$, for any $P,Q\in \s^2$, then it is clear that $\widehat{\overline{P}Q}=\pi-\widehat{PQ}$. So we deduce that
$$  \overline{\mathcal E _{d,e}}(F_1,F_2)=\mathcal E _{\pi-d,e}(F_1,F_2), \quad 
\overline{\mathcal H^1_{d,e}}(F_1,F_2)=\mathcal H^2 _{d,e}(F_1,F_2).$$

%Due to the geometry of the sphere, both spherical hyperbolas and spherical ellipses have two branches, one antipodal to the other. In the case of hyperbolas, both branches have the same parameters; however, in the case of ellipses, once the parameter $e$ is set, the branches have supplementary  parameters, that is, $d$ and $\pi -d$.

On the other hand, given a great circle $\mu$  and a fixed point $F_0$ (not belonging to $\mu$) in $S^2$, the set
$$ \mathcal P (F_0,\mu)=\{ P\in \s^2 : dist(P,F_0)=dist(P,\mu) \} $$
is called a \textit{spherical parabola} with focus $F_0$ and directrix $\mu$. 

Thanks to the closed geometry of the sphere, there are special relationships between the three types of conics defined above.

\begin{proposition}\label{Prop:conics} \phantom{salto}
	\begin{enumerate}
		\item Each branch of any spherical hyperbola with foci $F_1$ and $F_2$ is a spherical ellipse of foci $F_1$ and the antipodal point $\overline{F_2}$ of $F_2$ in $\s^2$. Specifically:
		$$ \mathcal H^1_{d,e}(F_1,F_2)=\mathcal E _{\pi/2+d,e}(F_1,\overline{F_2}), \quad
		  \mathcal H^2_{d,e}(F_1,F_2)=\mathcal E _{\pi/2-d,e}(F_1,\overline{F_2}).
		$$
		\item Any spherical parabola with focus $F_0$ and directrix $\mu$ is a branch of  a spherical hyperbola with foci $F_0$ and the pole
		$\widetilde{F_0}$  of $\mu$ in the opposite hemisphere to $F_0$ and parameter $d=\pi/4$. Specifically:
		$$
		\mathcal P (F_0,\mu)=\mathcal H^1_{\pi/4,e}(F_0,\widetilde{F_0}).
		$$
	\end{enumerate}
\end{proposition}\label{prop:relationconics}
\begin{proof}
	%We abbreviate $\widehat{PQ}=dist(P,Q)$, for any $P,Q\in \s^2$. 
	In order to prove part (1), we have that
	\begin{equation*}
	\widehat{PF_1}+\widehat{P\overline{F_2}}= \widehat{PF_1}+\pi-\widehat{PF_2}=\pi+2d \textrm{ or } \pi -2d,
\end{equation*}
according to the branch of the initial spherical hyperbola. To prove part (2), we reason as follows:
\begin{equation*}
	\widehat{PF_0}+\frac{\pi}{2}=dist(P,\mu)+\frac{\pi}{2}=\widehat{P\widetilde{F_0}} \Rightarrow 	\widehat{PF_0}-\widehat{P\widetilde{F_0}}=\frac{\pi}{2}.
\end{equation*}
\end{proof}

%\begin{figure}[h!] 
%	\centering 
%	\includegraphics[height=5cm]{defparabol.png} 
%	\caption{Vista en picado de la interpretación de la parábola esférica como rama de una hipérbola esférica.}
%	\label{fig:defparabola} 
%\end{figure}

Next, we deduce the equation, which we will call \textit{canonical}, of a spherical conic when we place the foci on the equator of the sphere.
\begin{theorem}\label{Th:eqnsgeo_conics} 
 The canonical equation of a spherical conic  $\mathcal C_{d,e}(F_1,F_2)$ (of parameters $d,e$) in geographical coordinates, i.e.\ $P=(\lambda,\varphi)\in \s^2$, when the foci are equidistant points from the zero-meridian on the equator, i.e.\ $F_1=(-e,0)$, $F_2=(e,0)$, is given by
 	\begin{equation}\label{eq:conic_geo}
 \mathcal C_{d,e} \equiv  \cos^2 \varphi \left( \frac{\cos^2 \lambda}{c_d^2/c_e^2}+\frac{\sin^2 \lambda}{s_d^2/s_e^2} \right)=1, \ 0<d,e<\pi/2, 
 \end{equation}
where $c_d=\cos d, s_d=\sin d, c_e=\cos e, s_e=\sin e \in (0,1)$.

\noindent $\mathcal C_{d,e}$ is a spherical ellipse if $d>e$; $\mathcal C_{d,e}$ is a spherical hyperbola if $d<e$, and  $\mathcal C_{d,e}$ is a spherical parabola if, in addition, $d=\pi/4$.
\end{theorem}
\begin{remark}\label{rm:eqbranches}
Since $c_{\pi-d}^2=c_d^2$ and $s_{\pi-d}^2=s_d^2$, we point out that \eqref{eq:conic_geo} includes both the spherical ellipse and its antipodal image when $d>e$.  In the same way, taking Proposition \ref{prop:relationconics}(2) into account, \eqref{eq:conic_geo} also includes the two branches of the spherical hyperbola producing the spherical parabola when $\pi/4=d<e$. See Figure \ref{fig:ed}.
\end{remark}
\begin{figure}[h!] 
	\centering
	% Primera columna
	\begin{subfigure}[t]{0.32\textwidth}
		\centering
		\includegraphics[width=\textwidth]{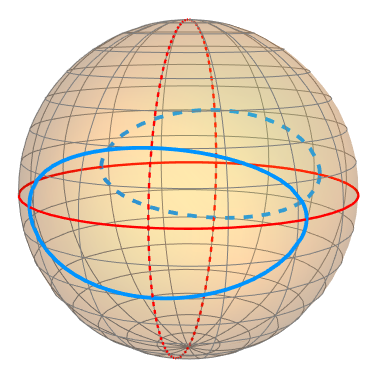}
		\caption{Ellipse, $d>e$.}
		\label{fig:elipseedfocos}
	\end{subfigure}
	\hfill
	% Segunda columna
	\begin{subfigure}[t]{0.32\textwidth}
		\centering
		\includegraphics[width=\textwidth]{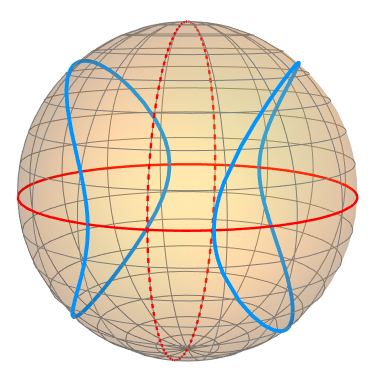}
		\caption{Hyperbola, $d<e$.}
		\label{fig:hiperbedfocos}
	\end{subfigure}
	\hfill
	% Tercera columna
	\begin{subfigure}[t]{0.32\textwidth}
		\centering
		\includegraphics[width=\textwidth]{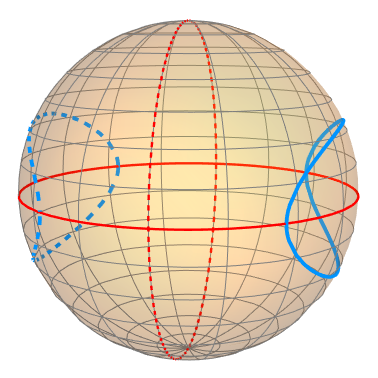}
		\caption{Parabola, $\pi/4=d<e$.}
		\label{fig:ecuadored}
	\end{subfigure}
	\caption{Spherical conics with foci on the equator in terms of the parameters $d$ and $e$.}
	\label{fig:ed}
\end{figure}

\begin{proof}
Let $l_1=\widehat{PF_1}$ and $l_2=\widehat{PF_2}$. Then $l_1+l_2=2d$, $d>e$, for a spherical ellipse, and 
$|l_1-l_2|=2d$, $d<e$, for a spherical hyperbola.

Applying twice spherical Pythagorean Theorem (see Figure \ref{fig:eqconics}), we deduce
\begin{equation}\label{eq:Pitagoras}
	\cos l_1= \cos(\lambda + e)\cos \varphi, \quad \cos l_2 = \cos (\lambda -e)\cos \varphi.
\end{equation}
\begin{figure}[h!] 
	\centering 
	\includegraphics[height=5cm]{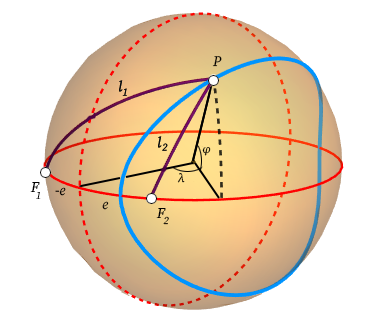} 
	\caption{Spherical conic with foci on the equator.}
	\label{fig:eqconics} 
\end{figure}

Using \eqref{eq:Pitagoras} and trigonometric formulas, we obtain:
\begin{equation}\label{eq.cosenos}
	\cos \left( \frac{l_1+l_2}{2} \right) 	\cos \left( \frac{l_1-l_2}{2} \right) = \frac{1}{2} \left( \cos \l_1 + \cos l_2 \right) =  \cos \lambda \cos \varphi \cos e.
\end{equation}
Then we have for a spherical ellipse that
\begin{equation}\label{eq:flaelipse}
	\cos d \cos (d-l_2)= \cos \lambda \cos \varphi \cos e, 
\end{equation}
and for a spherical hyperbola that
\begin{equation}\label{eq:flahiperbola}
	\cos (d\pm l_2) \cos d= \cos \lambda \cos \varphi \cos e.
\end{equation}
After manipulating \eqref{eq:flaelipse} and \eqref{eq:flahiperbola} and a straightforward computation using trigonometric identities, we finally reach \eqref{eq:conic_geo} in both cases. The case of a spherical parabola is included in the hyperbola case simply taking the particular value $d=\pi/4$.
\end{proof}

We point out that when $d=e$ in \eqref{eq:conic_geo}, we simply arrive at the geographical equation of the equator.

\begin{remark}\label{rm:conic_cart}
	Looking at \eqref{eq:conic_geo}, it is clear that the canonical equation of a spherical conic $\mathcal C_{d,e}$ (of parameters $d,e$) in Cartesian coordinates, i.e.\ $P=(x,y,z)\in \s^2$, when the foci are equidistant points from the zero-meridian on the equator, i.e.\ $F_1=(c_e,-s_e,0)$, $F_2=(c_e,s_e,0)$, is given by
	\begin{equation}\label{eq:conic_cartxy}
		\mathcal C_{d,e} \equiv  \frac{x^2}{c_d^2/c_e^2}+\frac{y^2}{s_d^2/s_e^2}=1, \ 0<d,e<\pi/2.
	\end{equation}
	And when the foci are equidistant points from the equator on the zero-meridian , i.e.\ $F_1=(c_e,0,-s_e)$, $F_2=(c_e,0,s_e)$, is given by
	\begin{equation}\label{eq:conic_cartxz}
		\mathcal C_{d,e} \equiv  \frac{x^2}{c_d^2/c_e^2}+\frac{z^2}{s_d^2/s_e^2}=1, \ 0<d,e<\pi/2.
	\end{equation}
Using the monotonicity  of the functions $\cos^2$ and $\sin^2$ in the interval $(0,\pi/2)$,
we have that either $c_d^2/c_e^2<1<s_d^2/s_e^2$ if, and only if, $d>e$, i.e. $	\mathcal C_{d,e}$ is an ellipse, or $c_d^2/c_e^2>1>s_d^2/s_e^2$ if, and only if, $d<e$, i.e.  $	\mathcal C_{d,e}$ is a hyperbola. If, in addition $d=\pi/4<e$, we get $c_e^2<1/2<s_e^2$ for a parabola. 
\end{remark}

\subsection{Spherical conics as intersection with elliptic cylinders} \label{Sect:ConicsIntersection}
Our next objective is to calculate the spherical angular momentum of a spherical conic in the canonical positions described in Remark \ref{rm:conic_cart} with respect to the equator $\xi_0$. Recall that, to do this, we are free to rotate the conic around the z-axis at our convenience (see Section \ref{Sect:KCurvesS2}).

Canonical equations \eqref{eq:conic_cartxy} and \eqref{eq:conic_cartxz} suggest a new analytical treatment of spherical conics. Specifically, we will study them in this section as intersections of the sphere $\s^2$ with two types of elliptic cylinders, that we call \textit{vertical} and \textit{horizontal} cylinders. The corresponding conics thus obtained which will be denoted by $\mathcal{CV}_{AB}$ and $\mathcal{CH}_{CD}$ respectively: 
	\begin{align}
	\mathcal{CV}_{AB}\equiv \frac{x^2}{A^2}+\frac{y^2}{B^2}=1, \ A,B > 0,
	\label{eq:cilindrosvert}
	\\
	\mathcal{CH}_{CD}\equiv \frac{x^2}{C^2}+\frac{z^2}{D^2}=1, \ C,D>0. 
	\label{eq:cilindroshoriz}
\end{align}
%\begin{figure}[h!]
%		\centering 
%		\begin{subfigure}[t]{0.37\textwidth}
%				\includegraphics[width=1\textwidth]{CV1.png} 
%				\label{fig:CV1}
%			\end{subfigure}
%		\hspace*{2cm}
%		\begin{subfigure}[t]{0.45\textwidth}
%				\includegraphics[width=1\textwidth]{CH1.png} 
%				\label{fig:CH1}
%			\end{subfigure}
%		\caption{Cónicas generadas por cilindros.}
%\end{figure}
Parameters $A$ and $B$ in \eqref{eq:cilindrosvert} (resp.\ $C$ and $D$ in \eqref{eq:cilindroshoriz}) are the lengths of the semi-axes of the base ellipse of the corresponding cylinder. Therefore, it is obvious that we will require $A\leq 1$ or $B\leq 1$ (resp.\ $C\leq 1$ or $D\leq 1$) as we are intersecting them with the unit sphere. 
\begin{remark}\label{rm:cone}
Using $x^2+y^2+z^2=1$ in \eqref{eq:cilindrosvert} and  \eqref{eq:cilindroshoriz}, we have that $\mathcal{CV}_{AB}$ and $\mathcal{CH}_{CD}$ can be also given as the intersection of $\s^2$ with an elliptic cone or a hyperbolic cylinder.
\end{remark}
On the other hand, it is easy to check that $\mathcal{CV}_{AA}$ produces $\xi_{\arcsin \sqrt{1-A^2}}$ and $\xi_{-\arcsin \sqrt{1-A^2}}$ (see Example \ref{ex:parallel}) and $\mathcal{CH}_{CC}$ gives $\eta_{\arctanh \sqrt{1-C^2}}$ and $\eta_{-\arctanh \sqrt{1-C^2}}$ (see Example \ref{ex:small circle}). On the other hand, recalling Example \ref{ex:great circle}, if $A=1>B$ (resp.\ $B=1>A$) we get the great circles $\mu_{\arcsin \sqrt{1-B^2}}$ and $\mu_{\pi-\arcsin \sqrt{1-B^2}}$ (resp.\ a $\pi/2$-rotation around $z$-axis of the great circles $\mu_{\arcsin \sqrt{1-A^2}}$ and $\mu_{\pi-\arcsin \sqrt{1-A^2}}$); if $C=1>D$ (resp. $D=1>C$) we obtain the great circles $\mu_{\arcsin D}$ and $\mu_{\pi-\arcsin D}$ (resp.\ the meridians $\lambda = \pm \arcsin\sqrt{1-C^2}$). 
See Figure \ref{fig:oldCircles}.
\begin{figure}[h!]
		\centering 
		\includegraphics[height=4.8cm]{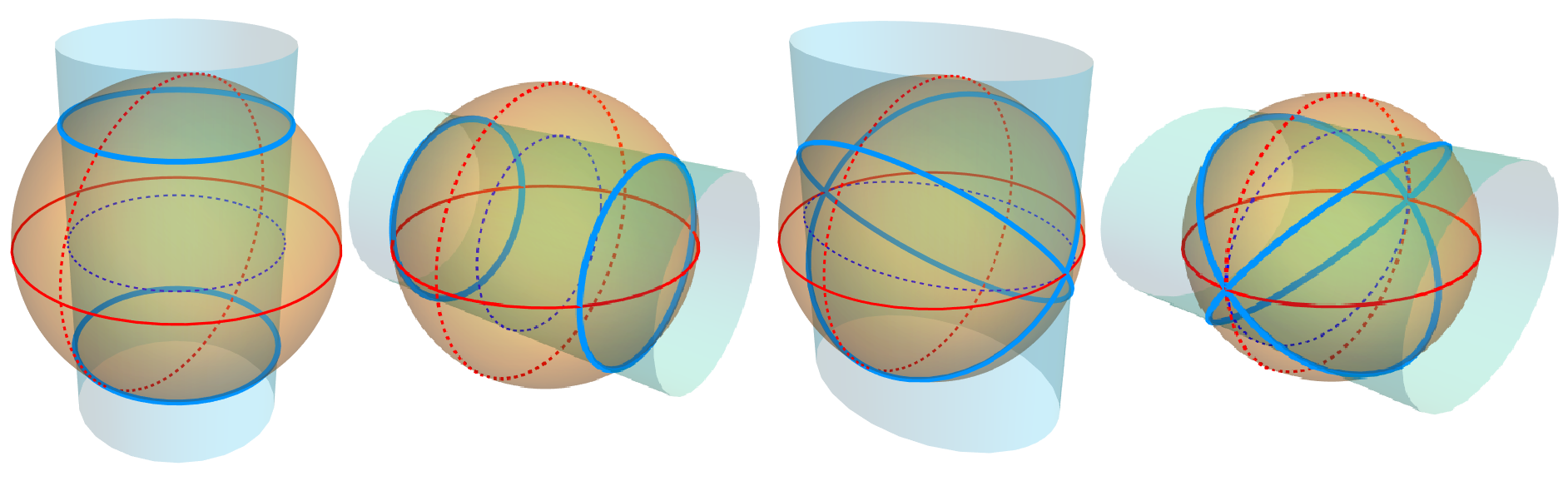} 
		\caption{Curves $\xi_{\varphi_0}$ of Example \ref{ex:parallel}, $\eta_\delta$ of Example \ref{ex:small circle}, and $\mu_\theta$ of Example \ref{ex:great circle}, as intersection with vertical and horizontal  cylinders.}
		\label{fig:oldCircles}
\end{figure}

Therefore, from now on, we will assume that $\min \{A,B\}<1$, $A\neq B$, $A\neq 1$, $B\neq 1$, and $\min\{C,D\}<1$, $C\neq D$, $C\neq 1$, $D\neq 1$, in order to study what we will call \textit{non-degenerate spherical conics}.

%To interpret vertical-type I and horizontal-type II cases geometrically, 
We should also note that some spherical conics can be given simultaneously as the intersection of the sphere with both vertical and horizontal cylinders, see Figure \ref{fig:both}. 
\begin{figure}[h!]
	\centering 
	\begin{subfigure}[t]{0.4\textwidth}
		\includegraphics[width=\textwidth]{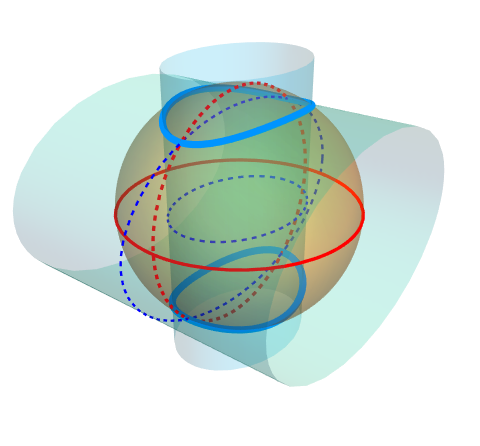} 
		\caption{$\mathcal{CV}_{AB}\equiv \mathcal{CH}_{CD}$}
	\end{subfigure}
	\begin{subfigure}[t]{0.4\textwidth}
		\includegraphics[width=\textwidth]{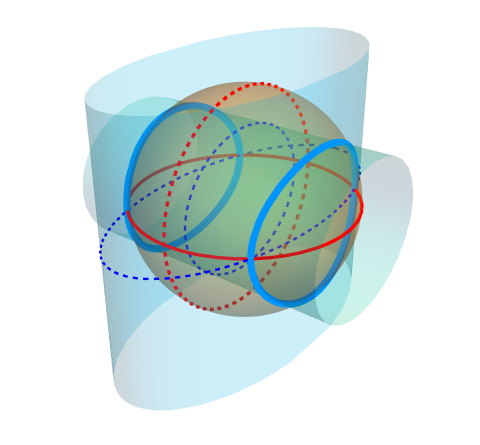} 
		\caption{$\mathcal{CH}_{CD}\equiv \mathcal{CV}_{AB}$}
	\end{subfigure}
	\caption{Spherical conics as the intersection of the sphere with both vertical and horizontal cylinders.}
	\label{fig:both}
\end{figure}

We analyze this situation using $x^2+y^2+z^2=1$ in \eqref{eq:cilindrosvert} and \eqref{eq:cilindroshoriz}. A simple computation gives that
\begin{equation}
	\label{eq:CDenfunciondeAB}
	\mathcal{CV}_{AB}\equiv \mathcal{CH}_{CD}, \	\text{ with } C^2=\dfrac{A^2(1-B^2)}{A^2-B^2},\ D^2=1-B^2, \text{ when } B<A, B<1,
\end{equation} 
and conversely
\begin{equation}
	\label{eq:ABenfunciondeCD}
	\mathcal{CH}_{CD}\equiv \mathcal{CV}_{AB}, \text{ with }	A^2=\dfrac{C^2(1-D^2)}{C^2-D^2},\ B^2=1-D^2, \text{ when } D<C, D<1.
\end{equation}

Now we are ready to identify the three essential types of non-degenerate spherical conics (in canonical positions) according to the entire possible range of variation of the parameters $A$ and $B$ in $\mathcal{CV}_{AB}$ and $C$ and $D$ in $\mathcal{CH}_{CD}$.
To do that, we will use \eqref{eq:CDenfunciondeAB} and \eqref{eq:ABenfunciondeCD}, a $\pi/2$-rotation around $z$-axis if necessary, and Remark \ref{rm:conic_cart} putting $A=c_d/c_e$, $B=s_d/s_e$ (resp.\ $C=c_d/c_e$, $D=s_d/s_e$).
\begin{enumerate}
	\item \textbf{Type I:} The spherical conics $\mathcal{CV}_{AB}^I\equiv A,B<1$, $A>B$ without restriction, are the same that $ \mathcal{CH}_{CD}^I \equiv D<1<C$, see Figure \ref{Fig:TypeI}. Therefore they are 
\textit{spherical hyperbolas,  whose foci are equidistant points from the equator on the zero-meridian} and whose parameters are given by
	\begin{align*}
		d=\arccos\left(C\sqrt{\dfrac{1-D^2}{C^2-D^2}}\right)=\arccos A, \quad 
		e=\arccos\left(\sqrt{\dfrac{1-D^2}{C^2-D^2}}\right)=\arccos \left( \sqrt{\dfrac{A^2-B^2}{1-B^2}}\right).
	\end{align*}
	We remark that the spherical parabolas appearing when $d=\pi/4$ correspond to the condition $C^2+D^2=2\, C^2D^2$ or, equivalently, $A^2=1/2>B^2$.
	
	\begin{figure}[h!]
	\centering % Centra la imagen
	\begin{subfigure}[t]{0.26\textwidth}
		\includegraphics[width=1\textwidth]{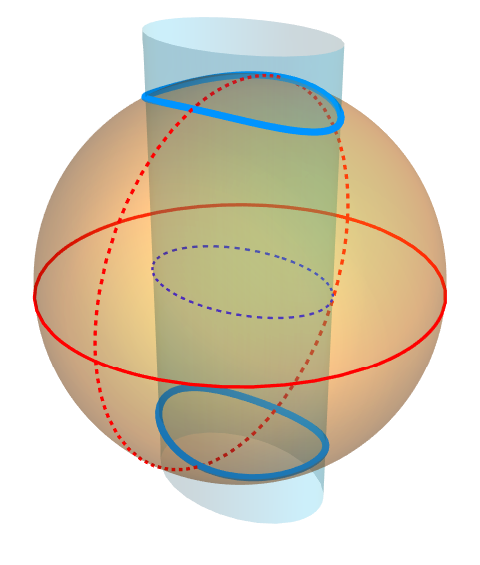} 
		\caption*{$A<B<1$}
		\label{fig:CV11}
	\end{subfigure}
	\hspace{0.5cm}
	\begin{subfigure}[t]{0.26\textwidth}
		\includegraphics[width=1\textwidth]{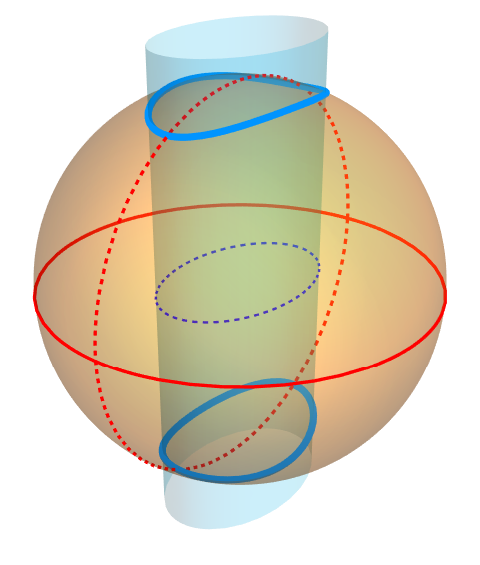} 
		\caption*{$1>A>B$}
		\label{fig:CV12}
	\end{subfigure}
	\begin{subfigure}[t]{0.42\textwidth}
		\includegraphics[width=1\textwidth]{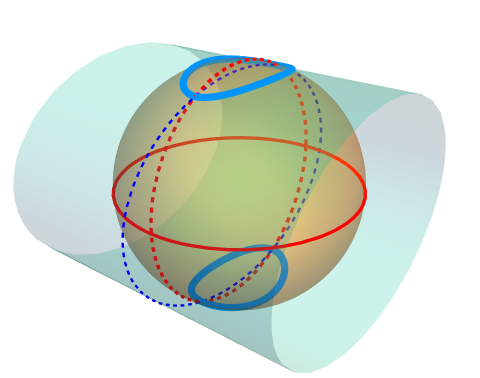} 
		\caption*{$D<1<C$}
		\label{fig:CH11}
	\end{subfigure}
	\caption{Type I non-degenerate spherical conics.}
	\label{Fig:TypeI}
\end{figure}
	
	\item \textbf{Type II:} The spherical conics $ \mathcal{CH}_{CD}^{II}\equiv D<C<1$ are the 
	same that $\mathcal{CV}_{AB}^{II} \equiv B <1<A$. Therefore they are \textit{spherical hyperbolas
	whose foci are equidistant points from the zero-meridian on the equator} and whose parameters are given by
	\begin{align*}
		d=\arccos\left(A\sqrt{\dfrac{1-B^2}{A^2-B^2}}\right)=\arccos C, \quad 
		e=\arccos\left(\sqrt{\dfrac{1-B^2}{A^2-B^2}}\right)=\arccos\left(\sqrt{\dfrac{C^2-D^2}{1-D^2}}\right).
	\end{align*}
	Recall that the spherical parabolas appearing when $d=\pi/4$ correspond to the condition $A^2+B^2=2A^2B^2$ or, equivalently $C^2=1/2>D^2$.
	We point out that if $A<1<B$ we get \textit{spherical ellipses whose foci are equidistant points from the zero-meridian on the equator}, but a $\pi/2$-rotation around $z$-axis leads the previous case. See Figure \ref{Fig:TypeII}.
	
\begin{figure}[h!]
	\centering % Centra la imagen
	\begin{subfigure}[t]{0.32\textwidth}
		\includegraphics[width=1\textwidth]{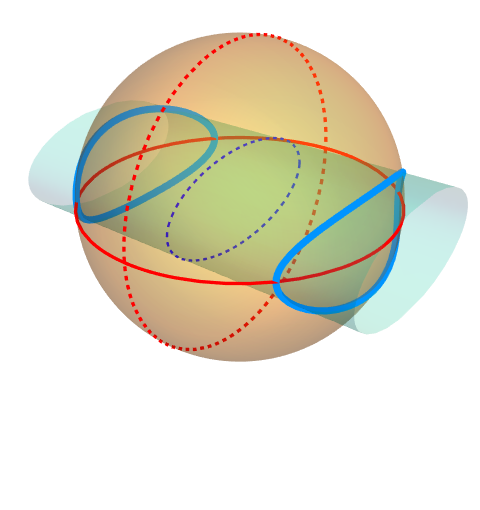} 
		\caption*{$D<C<1$.}
		\label{fig:CH21}
	\end{subfigure}
	\begin{subfigure}[t]{0.3\textwidth}
		\includegraphics[width=1\textwidth]{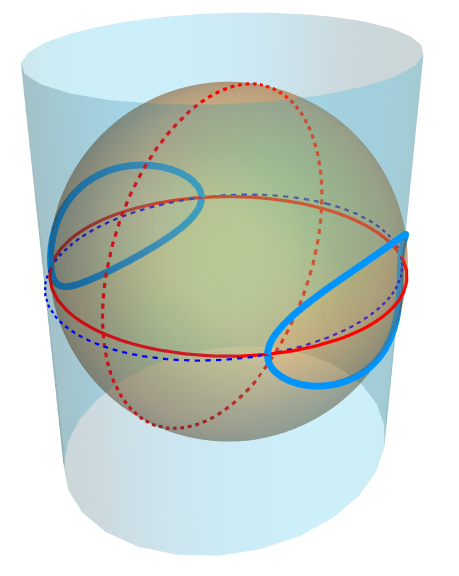} 
		\caption*{$B<1<A$.}
		\label{fig:CV21}
	\end{subfigure}
	\begin{subfigure}[t]{0.3\textwidth}
		\includegraphics[width=1\textwidth]{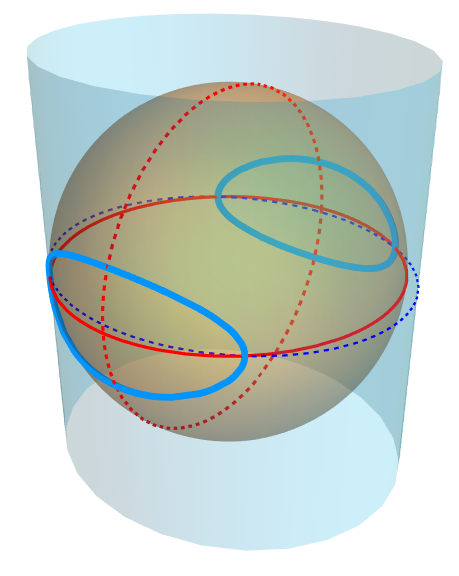} 
		\caption*{$A<1<B$.}
		\label{fig:CV22}
	\end{subfigure}
	\caption{Type II non-degenerate spherical conics.}
	\label{Fig:TypeII}
\end{figure}

	\item \textbf{Type III:} The spherical conics $ \mathcal{CH}_{CD}^{III}\equiv C<1<D$
	are \textit{spherical ellipses whose foci are equidistant points from the equator on the zero-meridian} and whose parameters are given by
	\begin{align*}
		d=\arccos\left(C\sqrt{\dfrac{D^2-1}{D^2-C^2}}\right), \quad
		e=\arccos\left(\sqrt{\dfrac{D^2-1}{D^2-C^2}}\right).
	\end{align*}
If $C<D<1$, it is easy to check that a $\pi/2$-rotation around $z$-axis brings us back to the previous case. See Figure \ref{Fig:TypeIII}.

\begin{figure}[h!]
	\centering % Centra la imagen
	\begin{subfigure}[t]{0.36\textwidth}
		\includegraphics[width=1\textwidth]{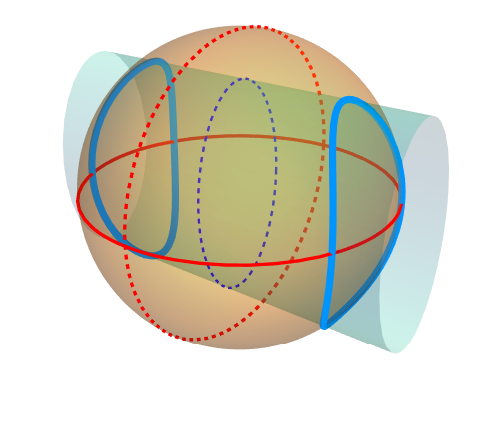} 
		\caption*{$C<D<1$.}
		\label{fig:CH31}
	\end{subfigure}
	\begin{subfigure}[t]{0.4\textwidth}
		\includegraphics[width=1\textwidth]{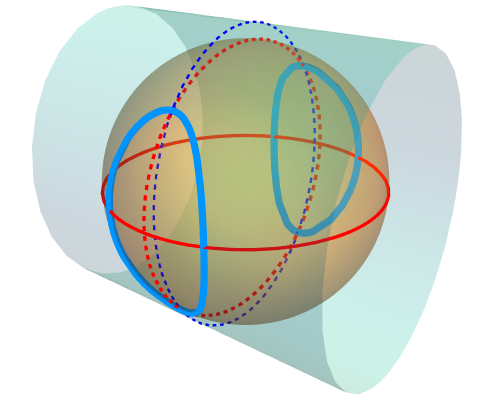} 
		\caption*{$C<1<D$.}
		\label{fig:CH32}
	\end{subfigure}
		\caption{Type III non-degenerate spherical conics.}
		\label{Fig:TypeIII}
\end{figure}
\end{enumerate}

\begin{remark}\label{rm:Perseo}
It is interesting to note that when projecting the non-degenerate spherical conics stereographically from the North Pole $(0,0,1)$ to the $xy$-plane, we find the quartics 
$(x^2+y^2)^2-2a  x^2-2b  y^2+1=0$,  with $a=\frac{C^2+C^2 D^2-2D^2}{C^2(1-D^2)}$ and $b=\frac{1+D^2}{1-D^2}$.
 They are bicircular quartics with a symmetry center, known as \textit{spiric curves of Perseus} in the literature, cf.\ \cite{F93}. Historically, these curves were defined as the sections of a (complex) torus by a plane parallel to its axis. See Figure \ref{Fig:Perseo}.
\end{remark}

	\begin{figure}[h!]
	\centering % Centra la imagen
	\begin{subfigure}[t]{0.31\textwidth}
		\includegraphics[width=1\textwidth]{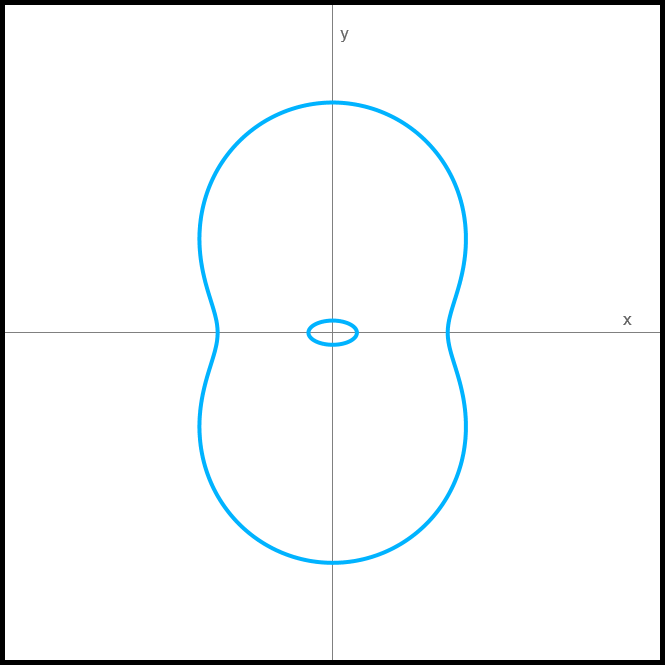} 
		\caption*{Type I.}
		
	\end{subfigure}
	\hspace*{0.1cm}
	\begin{subfigure}[t]{0.31\textwidth}
		\includegraphics[width=1\textwidth]{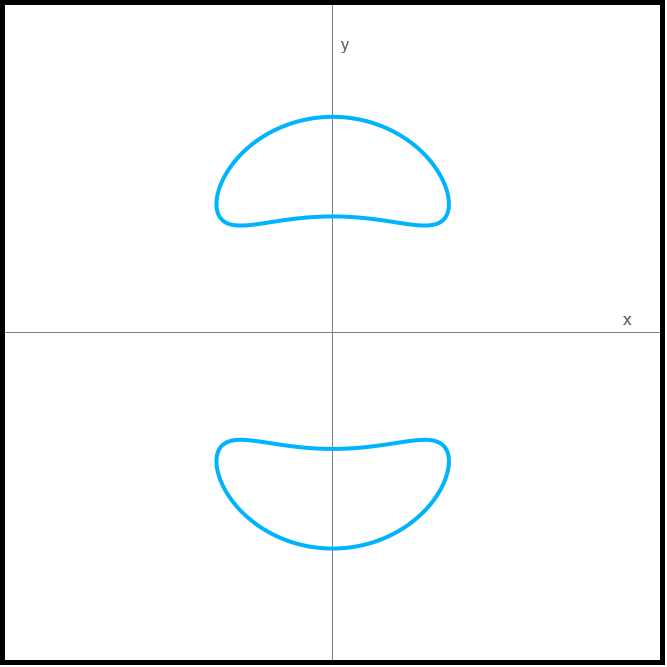} 
		\caption*{Type II.}
		
	\end{subfigure}
	\hspace*{0.1cm}
	\begin{subfigure}[t]{0.31\textwidth}
		\includegraphics[width=1\textwidth]{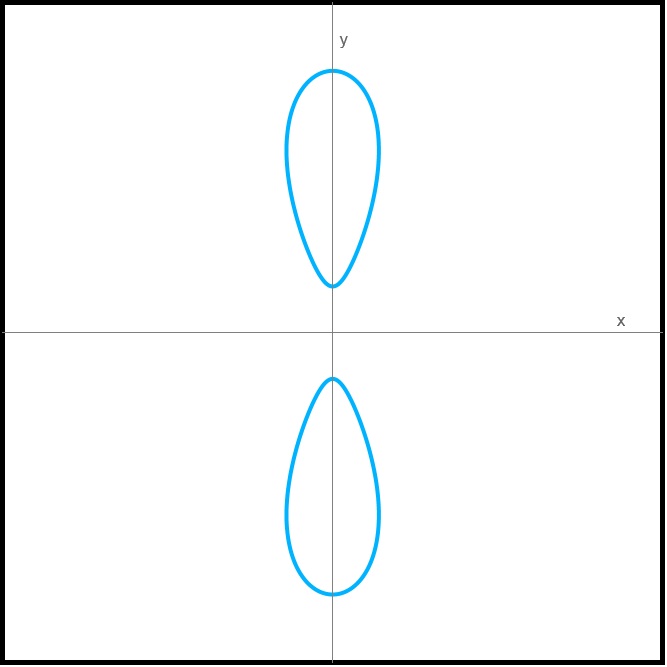} 
		\caption*{Type III.}
		
	\end{subfigure}
	
	\caption{Spiric curves of Perseus}
	\label{Fig:Perseo}
\end{figure}

We now proceed to calculate the spherical angular momentum of a non-degenerate spherical conic (in canonical position) in terms of its distance to the equator, finding out a common expression depending on two real parameters and distinguishing the three aforementioned types by the range of these constants.

\begin{proposition}
	\label{Prop:momentoangularesfericoABCD}
	
	The spherical angular momentum of a non-degenerate spherical conic $\mathcal C =(x,y,z)$ of any type in canonical position can be written as
	\begin{equation}\label{eq:K2muc}
		\K(z)^2=\dfrac{z^2}{\mu+c z^2}, 
	\end{equation}
	where $\mu \neq 0$ and $c \neq 0$ are given by 
\begin{equation}\label{eq:mucAB}
	\mu=\mu(A,B)=\dfrac{(A^2-1)(1-B^2)}{A^2B^2}, \ c=c(A,B)=\dfrac{1}{A^2B^2}, {\text \ for \ } \mathcal{CV}_{AB},
\end{equation}
and
	\begin{equation}\label{eq:mucCD}
	\mu=\mu(C,D)=\dfrac{D^4(1-C^2)}{C^2(1-D^2)^2}, \ c=c(C,D)=\dfrac{C^2-D^2}{C^2(1-D^2)^2}, {\text \ for \ } \mathcal{CH}_{CD}.
\end{equation}
	
	Moreover (see Figure \ref{Fig:moduli}),
	\begin{enumerate}
		\item for the spherical hyperbolae of type I: $\mu<0$, $\mu+c>1$, $\left(\mu-c+1\right)^2+4c\mu>0$; if, in addition, it is a spherical parabola, then also $2\mu+c=2$, $c>4$;
		\item for the spherical hyperbolae and ellipses of type II: $\mu>0$, $c>0$; if, in addition, it is a spherical parabola, then also $\mu+c=1$;
		\item for the spherical ellipses of type III: $\mu>0$, $c<0$.
	\end{enumerate}
\end{proposition}

\begin{figure}[h!]
	\centering % Centra la imagen
	\includegraphics[height=7.5cm]{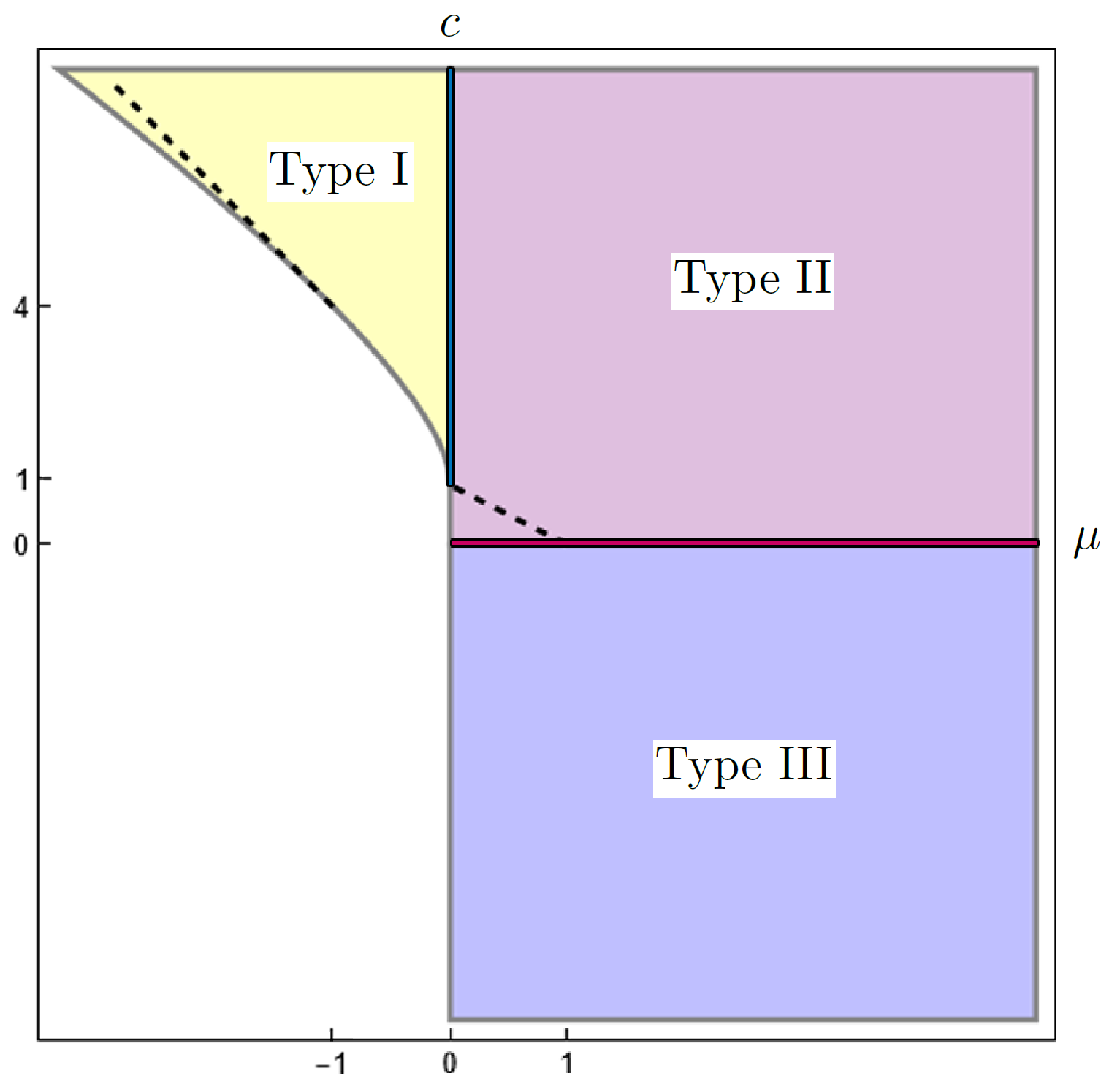} 
	\caption{Regions in  $\mu c$-plane corresponding to the three types of non-degenerate spherical conics. The dashed lines highlight spherical parabolas in the first two types.}
	\label{Fig:moduli}
\end{figure}

\begin{remark}\label{rm:K2}
	Since $\mathcal K$ is well defined up to sign (see Section \ref{Sect:KCurvesS2}), we have preferred to use the expression of $\mathcal K^2 $ in \eqref{eq:K2muc}. Moreover, it is an exercise to check the compatibility of \eqref{eq:mucAB} and \eqref{eq:mucCD} with \eqref{eq:CDenfunciondeAB} and \eqref{eq:ABenfunciondeCD}. 
\end{remark}

\begin{remark}\label{rm:casos0}
Using Example \ref{rm:Kcircles}, if $\mu =0$ and $c>1$ in \eqref{eq:K2muc} we recover the spherical angular momentum of the great circles $\mu_\theta$, $\theta \neq \pi/2$,  and if $c=0$ (and so $\mu >0$) in \eqref{eq:K2muc} we arrive at the spherical angular momentum of the small circles $\eta_\delta$, $\delta \neq 0$.
\end{remark}

\begin{proof} We choose the following parametrization for the spherical conic $\mathcal{CV}_{AB} $:
	\begin{equation*}
		\mathcal{CV}_{AB}(t)=\left(A\cos t,B\sin t, \sqrt{1-A^2\cos^2t-B^2\sin^2t}\right), \ t\in I \subseteq (-\pi,\pi].
	\end{equation*} %$A,B\in\R^+$ y $\min\{A,B\}<1$
Using that 
\begin{align*}
	\cos^2t=\dfrac{z^2-1+B^2}{B^2-A^2},\
	%\sin^2t=1-\dfrac{z^2-1+B^2}{B^2-A^2},
\end{align*}
a straightforward computation gives that 
\begin{align*}
	\K(t)=\frac{x'(t)y(t)-x(t)y'(t)}{\sqrt{x'(t)^2+y'(t)^2+z'(t)^2}}=\frac{-AB}{\sqrt{1+\frac{(A^2-1)(1-B^2)}{z^2}}} 
\end{align*}
and \eqref{eq:K2muc} follows with $\mu $ and $c$ given in \eqref{eq:mucAB}.

A similar computation can be done for the spherical conic $\mathcal{CH}_{CD} $ parameterized by
\begin{equation*} \mathcal{CH}_{CD}(t)=\left(C\cos t, \sqrt{1-C^2\cos^2t-D^2\sin^2t},D\sin t\right), \ t\in I \subseteq (-\pi,\pi],
\end{equation*} %$C,D\in\R^+$ y $\min\{C,D\}<1$
reaching \eqref{eq:K2muc}  with $\mu $ and $c$ given in \eqref{eq:mucCD}.
	
In order to prove the restrictions on parameters $\mu $ and $c$, first we can consider $ \mathcal{CH}_{CD}^I $, $D<1<C$, for type I spherical conics, and so formulas \eqref{eq:mucCD}. It is clear that $\mu <0$ since $C>1$. It is easy to check that $\mu +c >1/(1-D^2)$ and so $\mu + c >1$ since $D<1$. Finally, after a long computation, we get that $\left(\mu-c+1\right)^2+4c\mu=\frac{D^4}{C^4(1-D^2)^2}>0$. Moreover, using that $	C^2+D^2=2\, C^2D^2$ or $A^2=1/2>B^2$ for a spherical parabola, we get that $2\mu+c=2$. In this case, $c=2/(1-D^2)=2/B^2>4$.

Now we can take $ \mathcal{CV}_{AB}^{II} $, $B<1<A$, for type II spherical conics, and consider formulas \eqref{eq:mucAB}. It is obvious that $c>0$ and $\mu>0$ since $B<1<A$. Observe that the same happens if $A<1<B$. Using that $A^2+B^2=2\,A^2B^2$ or $C^2=1/2$ for a spherical parabola, it is easy to obtain $\mu+c=1$.

For type III spherical conics, we make use of $ \mathcal{CH}_{CD}^{III} $, $C<1<D$, and \eqref{eq:mucCD}. Then $\mu >0$ since $C<1$ and $c<0$ since $C<D$.	
\end{proof}

%--> Graphical summary
%
%	\begin{figure}[h!]
%	\centering % Centra la imagen
%	\includegraphics[width=0.7\textwidth]{clasifconicascurvasv2.png} % Ruta y tamaño
%	\caption{Curvas esféricas con momento angular $\K(z)=-\dfrac{z}{\sqrt{c z^2+\mu}}$,  $\mu,c\in\R$.}
%	\label{fig:clasifconicascurvas}
%\end{figure}

%\vspace{0.2cm}

%%%%%%%%%%%%%%%%%%%%%%%%%%%%%%%%%%%%%%%%%%%%%%%%%%%%%%%%%%%%%%%%%%%%%%%%%%%%%%%%%%%%%%%%%%%%

\section{Quadric surfaces of revolution in $\s^3$} \label{Sect:Quadrics}

Once the non-degenerate spherical conics have been determined in Section \ref{Sect:ConicsIntersection}, we proceed to introduce the natural notion of the \textit{non-degenerate quadric surfaces of revolution} in $\s^3$ by rotating them around the equator. Using  \eqref{eq:cilindrosvert}, \eqref{eq:cilindroshoriz}, \eqref{eq:CDenfunciondeAB} and \eqref{eq:ABenfunciondeCD}, we also include their implicit equations, which are completely similar to the implicit equations of the corresponding quadric surfaces of revolution in Euclidean 3-space (see e.g.\ \cite[Section 7.1]{CCC26}). 

\begin{definition}\label{def:quadric} %\phantom{-}
%The three types of non-degenerate quadric surfaces of revolution in $\s^3$ are:
\phantom{salto}
\begin{enumerate}
	\item[(I)] \textit{One-piece spherical hyperboloids of revolution}: $ S_{\mathcal{CH}_{CD}^I} $, i.e.\ $D<1<C$, or $ S_{\mathcal{CV}_{AB}^I} $, i.e.\ $B<A<1$,
	whose implicit equation is $x_3^2+x_4^2=\alpha^2+\beta^2\, x_2^2$,
	%$$ S_{\mathcal{CH}_{CD}^I}\equiv x_3^2+x_4^2=\alpha^2+\beta^2\, x_2^2, $$ 
	where 
	$$\alpha^2=\frac{D^2(C^2-1)}{C^2-D^2}=1-A^2, \  \beta^2=\frac{D^2}{C^2-D^2}=\frac{A^2-B^2}{B^2}.$$ 
	
	\noindent In particular,  \textit{spherical genus one paraboloids of revolution} appear when $C^2+D^2=2\,C^2D^2$ or $A^2=1/2>B^2$. In this case, $\alpha^2 =1/2$. See Figures \ref{fig:hiperboloideunapieza} and \ref{fig:paraboloideunapieza}.
	
	\begin{figure}[h!]
		\centering % Centra la imagen
		\begin{subfigure}[t]{0.3\textwidth} % La opción [H] fija la imagen en esta posición exacta
			\centering % Centra la imagen
			\includegraphics[width=0.99\textwidth]{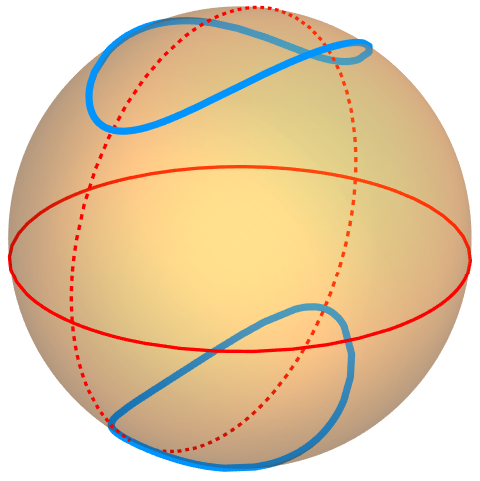} % Ruta y tamaño
		\end{subfigure}\hspace{1.3cm}
		\begin{subfigure}[t]{0.35\textwidth} % La opción [H] fija la imagen en esta posición exacta
			\centering % Centra la imagen
			\includegraphics[width=1\textwidth]{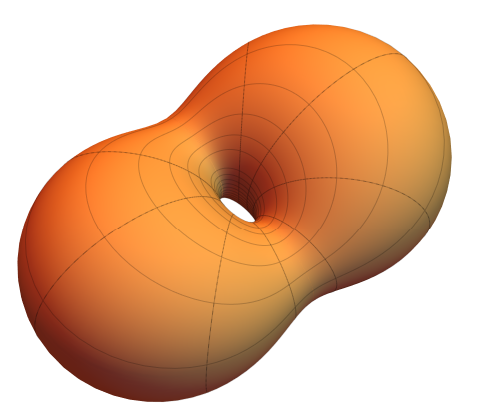} % Ruta y tamaño
		\end{subfigure}
		\caption{Generating curve and stereographic projection of a one-piece spherical hyperboloid of revolution.}  % Texto bajo la imagen
		\label{fig:hiperboloideunapieza}
	\end{figure}
	\begin{figure}[h!]
		\centering % Centra la imagen
		\begin{subfigure}[t]{0.3\textwidth} % La opción [H] fija la imagen en esta posición exacta
			\centering % Centra la imagen
			\includegraphics[width=0.99\textwidth]{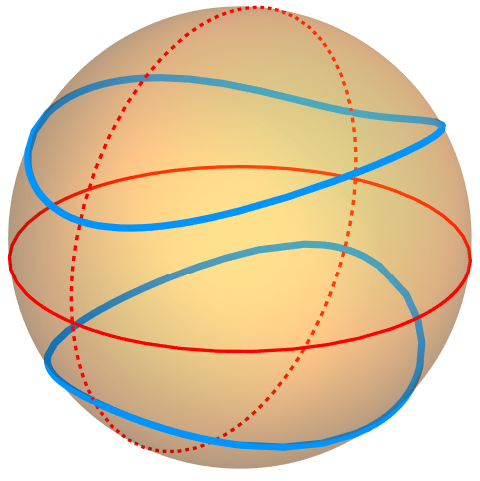} % Ruta y tamaño
		\end{subfigure}\hspace{1.3cm}
		\begin{subfigure}[t]{0.35\textwidth} % La opción [H] fija la imagen en esta posición exacta
			\centering % Centra la imagen
			\includegraphics[width=1\textwidth]{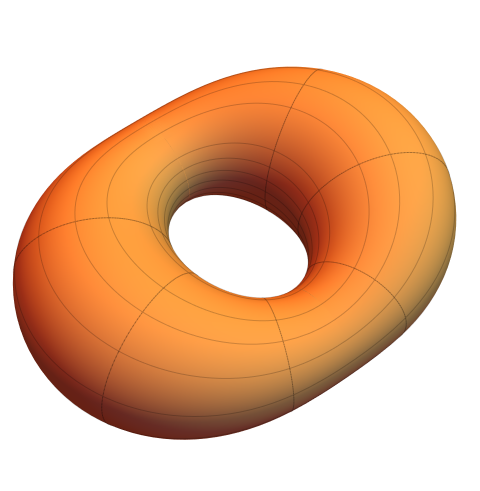} % Ruta y tamaño
		\end{subfigure}
		\caption{Generating curve and stereographic projection of a spherical genus one paraboloid of revolution.}  % Texto bajo la imagen
		\label{fig:paraboloideunapieza}
	\end{figure}
	
	\item[(IIa)] \textit{Two-piece spherical hyperboloids of revolution}: $ S_{\mathcal{CV}_{AB}^{II}} $, with  $B <1<A$, or $ S_{\mathcal{CH}_{CD}^{II}} $, i.e.\ $D<C<1$,  whose implicit equation is 
	$ x_3^2+x_4^2=-\alpha^2+\beta^2\, x_2^2$,
	%$$ S_{\mathcal{CV}_{AB}^{II}}\equiv x_3^2+x_4^2=-\alpha^2+\beta^2\, x_2^2, $$
	where 
	$$\alpha^2=A^2-1=\frac{D^2(1-C^2)}{C^2-D^2}, \ \beta^2=\frac{A^2-B^2}{B^2}=\frac{D^2}{C^2-D^2}.$$
	
	\noindent In particular, two copies of a \textit{spherical genus zero paraboloid of revolution} appear when  $A^2+B^2=2A^2B^2$ or $C^2=1/2>D^2$. In this case, $\beta^2 = 2 \alpha^2$. See Figures  \ref{fig:hiperboloidedospiezas} and \ref{fig:paraboloidedospiezas}.
	
		\item[(IIb)] \textit{Prolate spherical ellipsoids of revolution}: $ S_{\mathcal{CV}_{AB}^{II}} $, with $A <1<B$, whose implicit equation is $ x_3^2+x_4^2=\alpha^2-\beta^2\, x_2^2 $,
	%$$ S_{\mathcal{CV}_{AB}^{II}}\equiv x_3^2+x_4^2=\alpha^2-\beta^2\, x_2^2, $$
	where 
	$$\alpha^2=1-A^2, \ \beta^2=\frac{B^2-A^2}{B^2}.$$
	 It is clear that, up to a translation along the axis of revolution,  they coincide with the hyperboloids described in (IIa). See Figure \ref{fig:hiperboloidedospiezas}.
	 
	 \begin{figure}[h!]	
	 	\centering % Centra la imagen
	 	\begin{subfigure}[t]{0.3\textwidth} % La opción [H] fija la imagen en esta posición exacta
	 		\centering % Centra la imagen
	 		\includegraphics[width=0.99\textwidth]{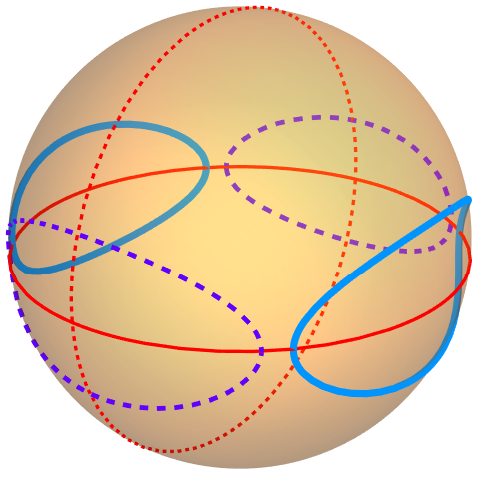} % Ruta y tamaño
	 	\end{subfigure}\hspace{1.3cm}
	 	\begin{subfigure}[t]{0.4\textwidth} % La opción [H] fija la imagen en esta posición exacta
	 		\centering % Centra la imagen
	 		\includegraphics[width=1\textwidth]{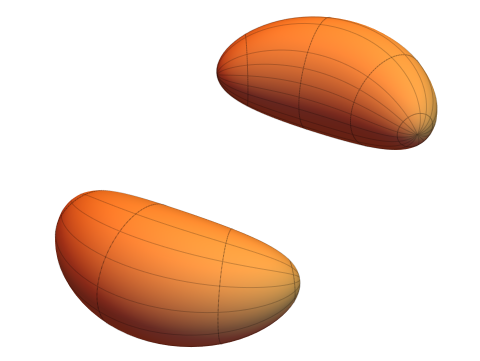} % Ruta y tamaño
	 	\end{subfigure}
	 	\caption{Generating curves and stereographic projection of a two-piece spherical hyperboloid of revolution or (two copies of) a prolate spherical ellipsoid of revolution.}  % Texto bajo la imagen
	 	\label{fig:hiperboloidedospiezas} % Etiqueta para referenciar la figura en el texto
	 \end{figure}
	 \begin{figure}[h!]
	 	\centering % Centra la imagen
	 	\begin{subfigure}[t]{0.3\textwidth} % La opción [H] fija la imagen en esta posición exacta
	 		\centering % Centra la imagen
	 		\includegraphics[width=0.99\textwidth]{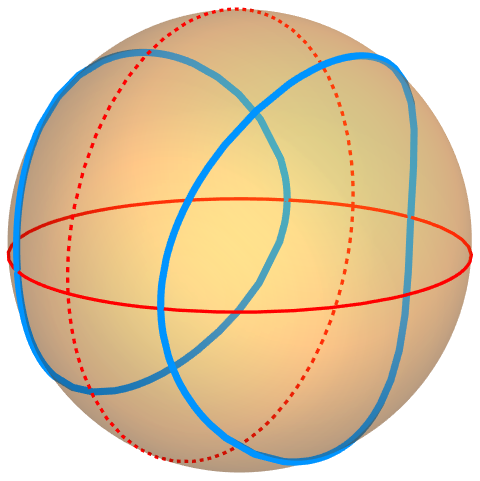} % Ruta y tamaño
	 	\end{subfigure}\hspace{1.3cm}
	 	\begin{subfigure}[t]{0.35\textwidth} % La opción [H] fija la imagen en esta posición exacta
	 		\centering % Centra la imagen
	 		\includegraphics[width=1\textwidth]{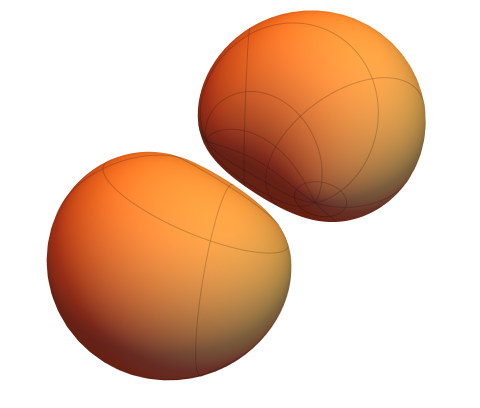} % Ruta y tamaño
	 	\end{subfigure}
	 	\caption{Generating curve and stereographic projection of (two copies of) a spherical genus zero paraboloid of revolution.}  % Texto bajo la imagen
	 	\label{fig:paraboloidedospiezas} % Etiqueta para referenciar la figura en el texto
	 \end{figure}

	\item[(III)] \textit{Oblate spherical ellipsoids of revolution}:  $S_{ \mathcal{CH}_{CD}^{III}}$, i.e.\ $C<1<D$, whose implicit equation is $x_3^2+x_4^2=\alpha^2-\beta^2\, x_2^2 $
	%$$ S_{ \mathcal{CH}_{CD}^{III}}\equiv x_3^2+x_4^2=\alpha^2-\beta^2\, x_2^2, $$
	where 
	$$\alpha^2=\frac{D^2(1-C^2)}{D^2-C^2}, \ \beta^2=\frac{D^2}{D^2-C^2}.$$ See Figure \ref{fig:elipsoide}.
	
		\begin{figure}[h!]
		\centering % Centra la imagen
		\begin{subfigure}[t]{0.3\textwidth} % La opción [H] fija la imagen en esta posición exacta
			\centering % Centra la imagen
			\includegraphics[width=0.99\textwidth]{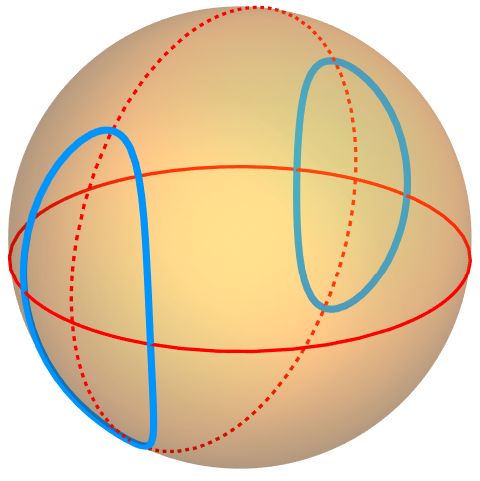} % Ruta y tamaño
		\end{subfigure}\hspace{1.3cm}
		\begin{subfigure}[t]{0.4\textwidth} % La opción [H] fija la imagen en esta posición exacta
			\centering % Centra la imagen
			\includegraphics[width=1\textwidth]{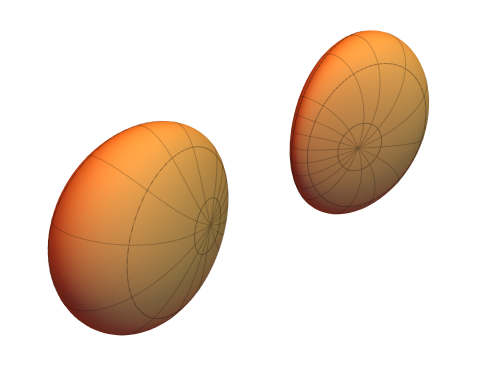} % Ruta y tamaño
		\end{subfigure}
		\caption{Generating curve and stereographic projection of (two copies of) a oblate spherical ellipsoid of revolution.}  % Texto bajo la imagen
		\label{fig:elipsoide} % Etiqueta para referenciar la figura en el texto
	\end{figure}
\end{enumerate}
\end{definition}

In view of the implicit equations that appear in Definition \ref{def:quadric}, the question naturally arises as to what surface it represents $x_3^2+x_4^2=2a x_2$, $a\neq 0$, which a priori could be thought of as appropriate for a spherical paraboloid of revolution. We will answer this question in the Proposition \ref{propo:fakeparaboloids}.

\begin{remark}\label{rm:impleqnsx1}
	The implicit equations of the non-degenerate quadric surfaces of revolution introduced in parts (I), (IIa) and (III) of Definition \ref{def:quadric} can be directly written in a common way as follows:
	\begin{equation}\label{eq:impl_eqns_x1}
		x_3^2+x_4^2=\widehat \alpha^2-\widehat \beta^2\, x_1^2, 
	\end{equation}
	where $\widehat \alpha = D$ and $ \widehat \beta = D/C$, and so $\widehat \alpha / \widehat \beta =C$. Then $\widehat \alpha , \widehat \beta < 1$, $\widehat \alpha / \widehat \beta > 1$ in (I), 
	$\widehat \alpha , \widehat \beta < 1$, $\widehat \alpha / \widehat \beta < 1$ in (IIa), 
	and $\widehat \alpha , \widehat \beta > 1$, $\widehat \alpha / \widehat \beta < 1$ in (III). The one for (IIb) is written as
	$	x_3^2+x_4^2=\widehat \beta^2\, x_1^2 -\widehat \alpha^2 $,
	with $\widehat \alpha^2= B^2-1$ and $\widehat \beta^2	=B^2/A^2-1 $, $A<1<B$, and so $\widehat \alpha / \widehat \beta < 1$.
	
	\noindent We emphasize that \eqref{eq:impl_eqns_x1} is also the equation of the \textit{degenerate} rotational quadric surfaces in $\s^3$. Specifically, $\widehat \alpha^2	=R_\delta^2 \leq 1$, $\widehat \beta^2= 1$, $\delta \geq 0$ for the totally umbilical spheres (see Example \ref{ex:tot_umb_S3}), and $\widehat \alpha^2	=\sin^2 \varphi_0$, $\widehat \beta^2= 0$, $ \varphi_0 \in (0,\pi/2)$, for the standard tori (see Example \ref{ex:standard_tori}), and $\widehat \alpha^2= \widehat \beta^2= \sin^2 \theta$, $ \theta \in (0,\pi)$, for the spherical moons (see Example \ref{ex:lun_esf_S3}).
\end{remark}

\begin{remark}\label{rm:Darboux}
	It is known that stereographic projections of standard tori (Example \ref{ex:standard_tori}) and spherical moons (Example \ref{ex:lun_esf_S3}) are examples of Dupin cyclides (see e.g.\ \cite{F93}). It is interesting to note that when projecting the non-degenerate quadric surfaces of revolution stereographically from the North Pole $(0,0,0,1)$ to the $xyz$-plane, we find the quartics 
	$\lambda(x^2+y^2+z^2)^2+L(x^2+y^2+z^2)+Q(x,z)=0$,  where $\lambda=C^2(D^2-1)$, $L=2C^2(D^2+1)>0$ and $Q(x,z)=C^2(D^2-1)-4D^2 x^2-4C^2 z^2$.
	They are specific non-degenerate \textit{Darboux cyclides} \cite{Z19} (i.e.\  stereographic projection of the intersection surface of a sphere and a quadric in $\R^4$). The circular arc structure of Darboux cyclides attracts the attention of geometric modeling community in applying them to contemporary free form architecture
	(see e.g.\ \cite{P12} and references therein).
\end{remark}

\begin{remark}\label{rm:Kquadric}
Using Corollary \ref{cor:Kkey}, Proposition \ref{Prop:momentoangularesfericoABCD} and Remark \ref{rm:casos0},	the non-degenerate quadric surfaces of revolution introduced in Definition \ref{def:quadric} are uniquely determined by the spherical angular momentum 
$
\mathcal K(z)=\dfrac{\pm z}{\sqrt{\mu+c z^2}}, 
$
$\mu\neq 0$, $c\neq 0$, under the following conditions:
\begin{enumerate}
	\item[(I)] $\mu<0$, $\mu+c>1$ and $\left(\mu-c+1\right)^2+4c\mu>0$ determine the one-piece spherical hyperboloids of revolution.
	If, in addition, $2\mu+c=2$, $c>4$, the spherical genus one paraboloids of revolution are fixed.
\item[(II)]  $\mu>0$ and  $c>0$ determine the two-piece spherical hyperboloids of revolution and the prolate spherical ellipsoids of revolution
If, in addition, $\mu+c=1$, the spherical genus zero paraboloids of revolution are fixed. 
\item[(III)] $\mu>0$ and $c<0$ determine the oblate spherical ellipsoids of revolution.
\end{enumerate}
\end{remark}

Now we show that the non-degenerate quadric surfaces of revolution in $\s^3$ introduced in Definition \ref{def:quadric} satisfy the same cubic Weingarten relation as the quadric surfaces of revolution in Euclidean 3-space (see \cite{CC22,CC23}) or Lorentz-Minkowski space (see \cite[Section 7.1]{CCC26}).

\begin{proposition}\label{prop:Kquadric}
	Any non-degenerate quadric surface of revolution in $\s^3$ satisfies the cubic Weingarten relation
	$k_{\rm m}=\mu \, k_{\rm p}^3$, for some constant $\mu \neq 0$.
\end{proposition}
\begin{proof} Using Remark \ref{rm:Kquadric}, we put 
$\K(z)=\pm \dfrac{z}{\sqrt{c z^2+\mu}}$, $\mu\neq 0$, $c\neq 0$, in \eqref{eq:kmkp} and so obtain 
\begin{equation*}
	k_m=\K'(z)=\pm \dfrac{\mu}{\sqrt{\left(cz^2+\mu\right)^3}}, \quad k_p=\dfrac{\K(z)}{z}=\pm \dfrac{1}{\sqrt{cz^2+\mu}},
\end{equation*}
Then it is clear that 	$k_{\rm m}=\mu \, k_{\rm p}^3$.
\end{proof}

We finish this section answering the question posed above about the \textit{natural} candidates for rotational  spherical paraboloids.

\begin{proposition}\label{propo:fakeparaboloids}
	The rotational surfaces in $\s^3$ given by the implicit equation $x_3^2+x_4^2=2a x_2$, $a \neq 0$, are generated by the sphero-cylindrical curves $x^2+(y+a)^2=1+a^2$ (see Figure \ref{fig:fake}). They satisfy the following sixth-degree Weingarten relation:
\begin{equation}\label{eq:K6th}
	(k_m-2k_p^3-3 k_p)^2=\frac{4}{1+a^2}(1+k_p^2)^3.
\end{equation}
\end{proposition}

	\begin{figure}[h!]
	\centering % Centra la imagen
	\begin{subfigure}[t]{0.3\textwidth} % La opción [H] fija la imagen en esta posición exacta
		\centering % Centra la imagen
		\includegraphics[width=0.99\textwidth]{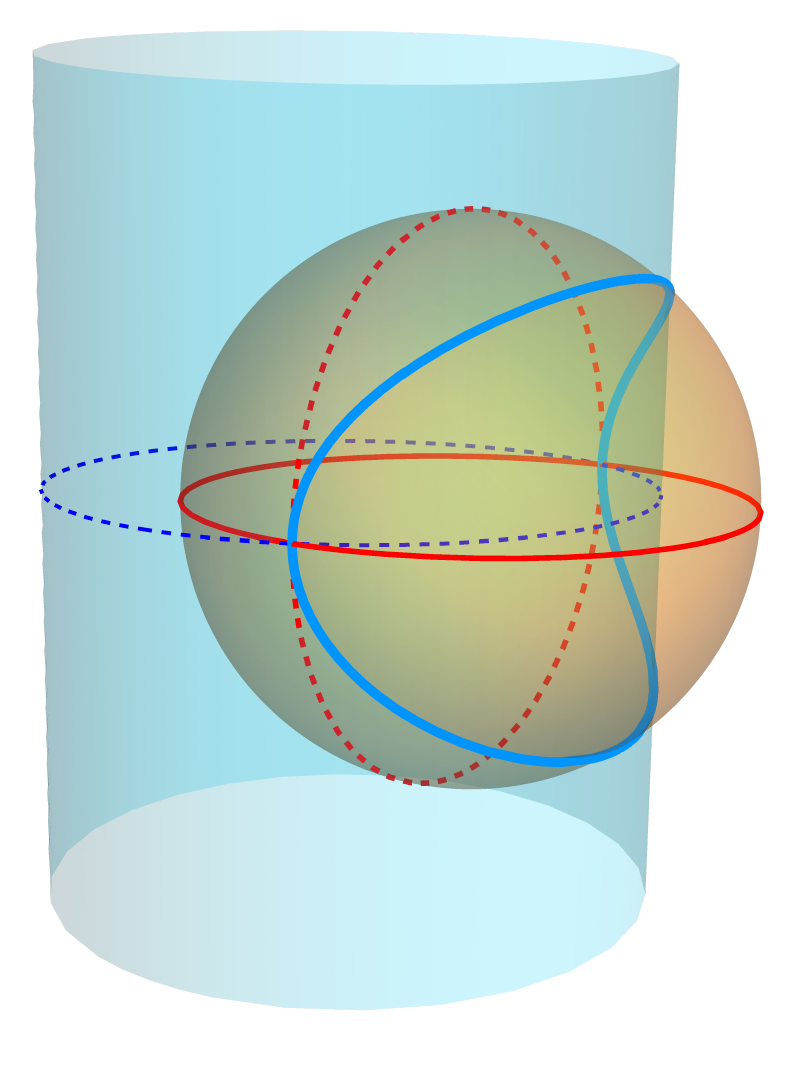} % Ruta y tamaño
	\end{subfigure}\hspace{1.3cm}
	\begin{subfigure}[t]{0.4\textwidth} % La opción [H] fija la imagen en esta posición exacta
		\centering % Centra la imagen
		\includegraphics[width=0.8\textwidth]{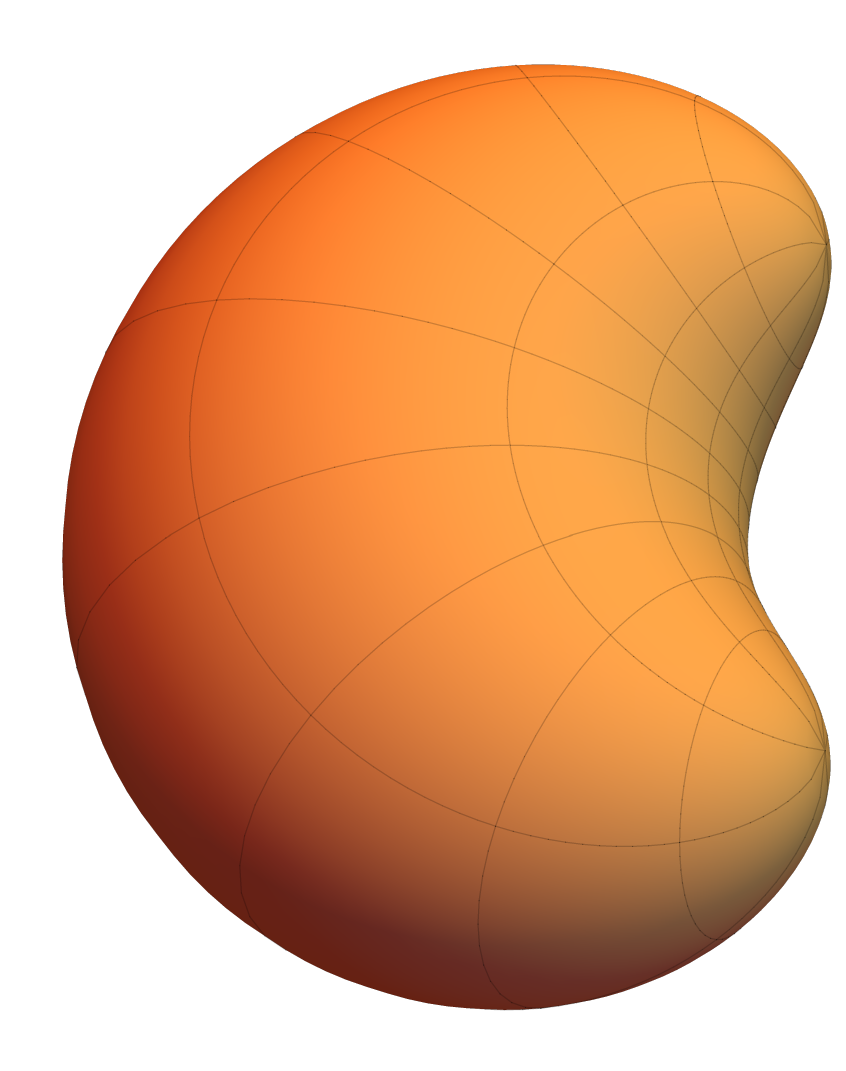} % Ruta y tamaño
	\end{subfigure}
	\caption{Generating curve and stereographic projection of the rotational surface $x_3^2+x_4^2=2a x_2$, $a \neq 0$.}  % Texto bajo la imagen
	\label{fig:fake} % Etiqueta para referenciar la figura en el texto
\end{figure}

\begin{proof}
Using \eqref{eq:paramXrot}, the generating curve $\xi_a$ of the rotational surface $x_3^2+x_4^2=2a x_2$, $a \neq 0$, is the intersection of $\s^2$ with the parabolic cylinder $z^2=2ay$. Taking into account that $x^2+y^2+z^2=1$, we get that $\xi_a \equiv x^2+(y+                                                                                                                                                                                                                                                                                                                                                                                                                                                                                                          a)^2=1+a^2$. We parameterize this curve by 
$$
x=\pm \frac{1}{2a}\sqrt{4a^2-4a^2 t^2-t^4}, \ y=\frac{t^2}{2a}, \ z=t, \quad t\in (-1,1), \, 4a^2-4a^2 t^2-t^4 \geq 0.
$$
From \eqref{eq:spherical momentum}, it is not difficult to reach that its spherical angular momentum is given by
\begin{equation}
\mathcal K (z)= \frac{z (z^2-2)}{\sqrt{4(a^2+z^2)-z^4}}.
\end{equation}
Then \eqref{eq:kmkp} leads to \eqref{eq:K6th}, after a long straightforward computation.
\end{proof}

%\begin{remark}\label{rm:Pinkall}
%	According to \cite{Pi85}, there is some?? relationship between certain?? rotational surfaces in $\s^3$ and \textit{Dupin cyclides} (see e.g.\ \cite{F93})).  
%\end{remark}

%\vspace{0.2cm}

%%%%%%%%%%%%%%%%%%%%%%%%%%%%%%%%%%%%%%%%%%%%%%%%%%%%%%%%%%%%%%%%%%%%%%%%%%%%%%%%%%%%%%%%%%%%

\section{Cubic rotational Weingarten surfaces in $\s^3$} \label{Sect:CubicW}

In the final section of the article, we establish our main result that characterizes degenerate and non-degenerate quadric surfaces of revolution  as the only rotational surfaces in $\s^3$  verifying a specific cubic Weingarten relation. % contained in Proposition \ref{prop:Kquadric}.

\begin{theorem}\label{Th:Main}
	The only rotational surfaces in the 3-sphere $\s^3$ satisfying the cubic Weingarten relation $k_{\rm m}=\mu \, k_{\rm p}^3$, for some $\mu\in\R$, are the following:
	\begin{enumerate}
		\item[(i)]  totally geodesic equatorial sphere (Example \ref{ex:tot_geod_S3}),
		\item[(ii)]  totally umbilical spheres (Example \ref{ex:tot_umb_S3}),
		\item[(iii)]  standard tori (Example \ref{ex:standard_tori}),
		\item[(iv)]  spherical moons (Example \ref{ex:lun_esf_S3}), 
		\item[(v)]  non-degenerate quadric surfaces of revolution (Definition \ref{def:quadric}).
%		\item Hiperboloide esférico de una pieza.
%		\item Hiperboloide esférico de dos piezas.
%		\item Elipsoide esférico.
	\end{enumerate}
\end{theorem}

\begin{remark}\label{rm:datosfinales}
We point out some facts about the possible values of constant $\mu $ in the cubic Weingarten relation $k_{\rm m}=\mu \, k_{\rm p}^3$ at each latter case:
\begin{enumerate}
	\item[(i)]  any $\mu \in \R$, since $k_{\rm m}=0= k_{\rm p}$ for the totally geodesic equatorial sphere $\s^2 \hookrightarrow \s^3$;
	\item[(ii)]  $\mu >0$ ($\mu=1/\sinh^2 \delta$), since $k_{\text m}=-\sinh \delta=	k_{\text p}$ for the totally umbilical sphere $\s^2(R_\delta) \hookrightarrow\s^3$, $0<R_\delta = \sech \delta < 1$, $\delta > 0$;
	\item[(iii)] $\mu <0$ ($\mu=-\tan^4 \varphi_0$), since $k_{\text m}=\tan \varphi_0 $ and $ k_{\text p}=-\cot \varphi_0 $ for the standard torus $\s^1(\cos\varphi_0)\times \s^1\left( \sin \varphi_0 \right) \hookrightarrow\s^3$, $\varphi_0 \in (0,\pi/2)$; 
	\item[(iv)]  $\mu =0$, since $k_{\rm m}=0$ for any spherical moon $S_{\mu_\theta}$, $\theta \in (0,\pi)$;
	\item[(v)]  $\mu \neq 0$, for the non-degenerate quadric surfaces of revolution (see Remark  \ref{rm:Kquadric}).
	%		\item Hiperboloide esférico de una pieza.
	%		\item Hiperboloide esférico de dos piezas.
	%		\item Elipsoide esférico.
\end{enumerate}
We remark that, in cases (ii) and (iii), the value of $\mu$ completely determines the corresponding spherical surface. In case (v), in addition to $\mu\neq 0$, another real non-null constant $c$ (see the proof of Theorem \ref{Th:Main}) and formulas \eqref{eq:IsolAB}, \eqref{eq:muDC}, \eqref{eq:D+} and \eqref{eq:D-} deduced there are involved.
\end{remark}

\begin{proof}
	Using Remark \ref{rm:datosfinales} and Proposition \ref{prop:Kquadric}, all the surfaces in the statement of Theorem \ref{Th:Main} satisfy the cubic relation $k_{\rm m}=\mu \, k_{\rm p}^3$, for some $\mu\in\R$. We now prove that they are the only ones.
	
	Assume then that a rotational surface in $\s^3$ verifies $k_{\rm m}=\mu \, k_{\rm p}^3$, $\mu\in\R$. Before applying Corollary \ref{cor:edoW}, we must consider the case that $z$ is constant, which leads to (iii) by Example \ref{ex:standard_tori}.
	 Otherwise, using \eqref{eq:kmkp}, the above cubic Weingarten relation translates into the separable o.d.e. 
\begin{equation}\label{eq:odeK}
 \mathcal K' = \mu \, \mathcal K^3 /z^3. 
\end{equation}
	If $\mu =0$, we have that $\mathcal K $ is constant and we arrive at (i) and (iv) by Remark \ref{rm:Ksimple}. When $\mu \neq 0$, the constant solution $\mathcal K \equiv 0$ of \eqref{eq:odeK} gives (i)  again. The general solution of \eqref{eq:odeK} is written as 
\begin{equation}\label{eq:Kth}
	 \K(z)=\pm \dfrac{z}{\sqrt{\mu + c z^2}}, 
\end{equation}
	where $c\in \R$ is the integration constant. If $c=0$, we obtain (ii) from Remark \ref{rm:Ksimple} once again. 
	
	Now we deduce the a priori conditions on $\mu\neq 0$ and $c\neq 0$ in \eqref{eq:Kth}. % in order to get (v). 
	If $\mu>0$, then $c>0$ or $c<0$. But when $\mu <0$, we can find several new restrictions:
	It is clear from \eqref{eq:Kth} that if $\mu <0$ then $c>0$. %, and if $c <0$ then $\mu >0$.
	Using Remark \ref{rm:algorithm}, we know that not only 
	$\mathcal K(z)^2 <1 $ but also 
	\begin{equation}\label{eq:Kineq}
		\mathcal K(z)^2 +z^2 <1.
	\end{equation}
		Then we deduce that $\mu + (c-1)z^2 > 0$. So if $\mu <0$ we have that $c>1$. In addition, using that $z^2\leq 1$, we can conclude that $\mu + c > 1$ if $\mu <0$. Moreover, from \eqref{eq:Kineq} also using that $\mu + c z^2>0$, we get that $ -cz^4+(c-\mu-1)z^2+\mu>0$. When $\mu <0$ (and consequently $c>1$), the last inequality necessarily implies that $\left(\mu-c+1\right)^2+4c\mu>0$. In conclusion, when $\mu\neq 0$ and $c\neq 0$, we have arrived at the same restrictions on them as those set out in Remark \ref{rm:Kquadric}.
		
To arrive at (v), given $\mu\neq 0$ and $c\neq 0$, we need to find suitable $A>0$ and $B>0$ (or $C>0$ and $D>0$) such that the equations in \eqref{eq:mucAB} (or in \eqref{eq:mucCD}) are satisfied. We have two possible paths:

\begin{enumerate}
\item 	
On the one hand, using \eqref{eq:mucAB}, we can write $\mu = \frac{S-P-1}{P} $, where $S=A^2+B^2$ and $P=A^2B^2$, and so
\begin{equation}\label{eq:SP}
	S=\frac{\mu + c+ 1}{c}, \quad P=\frac{1}{c} .
\end{equation}
Then we know that $A^2$ and $B^2$ must be positive roots of $X^2-SX+P=0$. We have that $S^2-4P=\left( (\mu-c+1)^2+4c\mu\right)/c^2>0$. Then we can take
\begin{equation}\label{eq:IsolAB}
	A^2=\frac{\mu+c+1+ \sqrt{(\mu-c+1)^2+4c\mu}}{2c}, \quad B^2=\frac{\mu+c+1- \sqrt{(\mu-c+1)^2+4c\mu}}{2c}, 
\end{equation}
distinguishing between the following cases:
\begin{enumerate}
	\item[(I)] $\mu<0$, $\mu+c>1$ and $\left(\mu-c+1\right)^2+4c\mu>0$: Then it is easy to check that 
$1> A^2 > B^2 >0$ since $\mu < 0$ and $c>0$.
	Therefore \eqref{eq:IsolAB} leads to the one-piece spherical hyperboloids of revolution (see Definition \ref{def:quadric}(I)).
	We observe that when $2\mu+c=2$, $c>4$, then $A^2=1/2$ from \eqref{eq:IsolAB}.
	
	\item[(II)]  $\mu>0$ and  $c>0$: Then it is easy to check that 
	$0< B^2 < 1 < A^2$.	
Then \eqref{eq:IsolAB}	 leads to the two-piece spherical hyperboloids of revolution (see Definition \ref{def:quadric}(IIa)).
Looking at \eqref{eq:SP}, we observe that when $\mu+c=1$, then $A^2+B^2=2A^2B^2$. 
To get the prolate spherical ellipsoids of revolution of Definition \ref{def:quadric}(IIb), we must simply interchange the definitions of $A^2$ and $B^2$ in \eqref{eq:IsolAB}.
\end{enumerate}		

\item On the other hand, from  \eqref{eq:mucCD}, we obtain that
\begin{equation}\label{eq:muDC}
\mu+D^2c=\frac{D^2}{1-D^2}, \quad C^2= \frac{D^2}{1-c(1-D^2)^2}.
\end{equation}
The first equation of \eqref{eq:muDC} is equivalent to 
$cD^4 +(\mu-c+1)D^2-\mu =0$. Thus $D^2$ must be a positive root of the second order polynomial 
$cX^2 +(\mu-c+1)X-\mu$. Then we deduce that
\begin{equation}\label{eq:D+}
D^2	=\frac{c-\mu-1 + \sqrt{(\mu-c+1)^2+4c\mu}}{2c}, {\rm \ if \ } \mu < 0, \, c>0 {\rm \ or \ } \mu > 0, \, c>0, 
\end{equation}
and
\begin{equation}\label{eq:D-}
D^2	=\frac{c-\mu-1 - \sqrt{(\mu-c+1)^2+4c\mu}}{2c}, {\rm \ if \ }   \mu >0, \, c<0.
\end{equation}
The corresponding value of $C^2$ is given by the second equation of \eqref{eq:muDC} putting $D^2$ of \eqref{eq:D+} or \eqref{eq:D-} as appropriate.

Then we must distinguish between the following cases:
\begin{enumerate}
	\item[(I)] $\mu<0$, $\mu+c>1$ and $\left(\mu-c+1\right)^2+4c\mu>0$: It is not difficult to check that 
	$0< D^2 < 1 < C^2 $.
	Therefore \eqref{eq:muDC} and \eqref{eq:D+} lead to the one-piece spherical hyperboloids of revolution, see Definition \ref{def:quadric}(I).
	We point out that when $2\mu+c=2$, $c>4$, then $C^2+D^2=2C^2 D^2$  from \eqref{eq:muDC}.
	
	\item[(II)]  $\mu>0$ and  $c>0$: Now we arrive at $0< D^2 < C^2 < 1 $.
	Then \eqref{eq:muDC} and \eqref{eq:D+}	produce the two-piece spherical hyperboloids of revolution, see Definition \ref{def:quadric}(IIa). Moreover, it is an exercise to prove $C^2=1/2$ when $\mu+c=1$.
	
	\item[(III)] $\mu>0$ and $c<0$: Here we get that $0< C^2 < 1 < D^2$. Then \eqref{eq:muDC} and \eqref{eq:D-} provide the oblate spherical ellipsoids of revolution, see Definition \ref{def:quadric}(III).
\end{enumerate}		
\end{enumerate}		
\end{proof}

\begin{remark}\label{rm:integral}
	Once the spherical angular momentum \eqref{eq:Kth} has been determined in the proof of Theorem \ref{Th:Main}, we could have tried to apply the algorithm described in Remark \ref{rm:algorithm} to recover the corresponding generatrix curve. The main obstacle to this procedure is that elliptic integrals appear when trying to obtain the arc parameter $s$ in terms of $z$. However, by eliminating $s$ and using the change of variable $z=\sin \varphi $, we would arrive at the following equation in geographic coordinates:
	\begin{equation}\label{eq:coorgeo1}
		d\lambda = \pm \frac{\tan \varphi \, d\varphi}{\sqrt{\cos^2 \varphi (\mu + c \, \sin^2 \varphi )-\sin^2 \varphi}}
	\end{equation}
	On the other hand, using \eqref{eq:cilindroshoriz}, it is easy to check that the equation in geographical coordinates for the conics $\mathcal{CH}_{CD}$  is given by
	\begin{equation}\label{eq:coorgeo2}
		d\lambda = \pm \frac{C(1-D^2)\tan \varphi \, d\varphi}{\sqrt{D^2-\sin^2 \varphi}\sqrt{C^2 \sin^2 \varphi +D^2 \cos^2 \varphi -C^2 D^2}}
	\end{equation}
Now, it is not difficult to verify that \eqref{eq:coorgeo1} and \eqref{eq:coorgeo2} are compatible precisely through \eqref{eq:mucCD}.
\end{remark}

An immediate consequence of Theorem \ref{Th:Main} and Remark \ref{rm:datosfinales} is the following result, that can be considered as the spherical version of \cite[Theorem 2]{CC22} in Euclidean space $\E^3$ or \cite[Theorem 7.3]{CCC26} in Lorentz-Minkowski space $\L^3$.

\begin{corollary}\label{cor:caracter}
	The totally umbilical spheres, the standard tori  and the non-degenerate quadric surfaces of revolution are the only rotational surfaces in $\s^3$ satisfying the cubic Weingarten relation $k_{\rm m}=\mu \, k_{\rm p}^3$, $\mu \neq 0$.
\end{corollary}

\vspace{0.5cm}

%%%%%%%%%%%%%%%%%%%%%%%%%%%%%%%%%%%%%%%%%%%%%%%%%%%%%%%%%%%%%%%%%%%%%%%%%%%%%%%%%%%%%%%%%%%%

\medskip

\noindent\textbf{Acknowledgments.} 
I.\ Castro is partially supported by the State Research Agency (AEI) via the grant no.\ PID2022-142559NB-I00, and the “Maria de Maeztu” Unit of Excellence IMAG, reference CEX2020001105-M, funded by MICIU/AEI/10.13039/501100011033 and ERDF/EU, Spain.

%%%%%%%%%%%%%%%%%%%%%%%%%%%%%%%%%%%%%%%%%%%%%%%%%%%%%%%%%%%%%%%%%%%%%%%%%%%%%%%%%%%%%%%%%%%%

%$$
%\frac{x^2}{a^2}+\frac{y^2}{b^2}=1
%$$
%
%$$
%\frac{x^2}{a^2}-\frac{y^2}{b^2}=1
%$$
%
%$$
%\frac{y^2}{a^2}-\frac{x^2}{b^2}=1
%$$
%
%$$
%y=4a x^2
%$$

%\newpage


\begin{thebibliography}{1}\bibliographystyle{alpha}

\bibitem{AL15} B.~Andrews and H.~Li.
{\em Embedded constant mean curvature tori in the three-sphere}.
J.\ Diff.\ Geom.\  {\bf 99} (2015), 169--189.

%\bibitem{BG94} 
%B.~van-Brunt and K.~Grant,
%{\it Hyperbolic Weingarten surfaces},
%Math.\ Proc.\ Cambridge Philos.\ Soc.\, \textbf{116} (1994), 489--504.


\bibitem{Br13a}  S.~Brendle.
{\em Embedded minimal tori in $\s^3$ and the Lawson Conjecture}.
Acta Math.\  {\bf 211} (2013), 177--190.

\bibitem{Br13b}  S.~Brendle.
{\em Minimal surfaces in $\s^3$: a survey of recent results}.
Bull.\ Math.\ Sci.\ {\bf 3} (2013), 133–-171.

\bibitem{Br13c}  S.~Brendle.
{\em Alexandrov immersed minimal tori in $\s^3$}.
Math.\ Res.\ Lett.\  {\bf 20} (2013), 459–-464.

\bibitem{Br14}  S.~Brendle.
{\em Embedded Weingarten tori in $\s^3$}.
Adv.\ Math.\   {\bf 257} (2014), 462–-475.

%\bibitem{BF71} P.F.~Byrd and M.D.~Friedman,
%{\em Handbook of elliptic integrals for engineers and physicists},
%Springer Verlag, Berlin, 1971.

\bibitem{CC22} P.~Carretero and I.~Castro. 
{\em A new approach to rotational Weingarten surfaces.} 
Mathematics {\bf 2022} 10(4), 578; https://doi.org/10.3390/math10040578

\bibitem{CC23} P.~Carretero and I.~Castro. 
{\em A Geometric Characterization of The Quadric Surfaces of Revolution.} 
Rom.\ J.\ Math.\ Comput.\ Sci.\ {\bf 15} (2023), 68--74. 

%\bibitem{CC24} P.~Carretero and I.~Castro. 
%{\em Rotational surfaces with prescribed curvatures.} 
%Diff.\ Geom.\ Appl.\ {\bf 101}, 2025, 102298; https://doi.org/10.1016/j.difgeo.2025.102298

%\bibitem{CCC25} P.~Carretero, I.~Castro and I.~Castro-Infantes. {\em Quadratic Rotational Weingarten Surfaces.} 
%Rom.\ J.\ Math.\ Comput.\ Sci.\ {\bf 13} (2025), 76--82. 

\bibitem{CCC26} P.~Carretero, I.~Castro and I.~Castro-Infantes. 
{\em Rotational Weingarten surfaces in Lorentz-Minkowski space.} 
Preprint arXiv:2512.10423 [math.DG].

%\bibitem{CCI16} I.~Castro and I.~Castro-Infantes.
%{\em Plane curves with curvature depending on distance to a line.}
%Diff.\ Geom.\ Appl.\ {\bf 44} (2016), 77--97.

%\bibitem{CCI18} I.~Castro, I.~Castro-Infantes and J.~Castro-Infantes.
%{\em Curves in Lorentz-Minkowski plane: elasticae, catenaries and grim-reapers}.
%Open Math.\ \textbf{16} (2018), 747--766.
%%
%\bibitem{CCI20} I.~Castro, I.~Castro-Infantes and J.~Castro-Infantes.
% {\em Curves in Lorentz-Minkowski plane with curvature depending on their position.}
%Open Math.\ \textbf{18} (2020), 749--770.

%\bibitem{CCIs20} I.~Castro, I.~Castro-Infantes and J.~Castro-Infantes.
%{\em On a Problem of David Singer about Prescribing Curvature for Curves.}
%Geom.\ Integrability \& Quantization {\bf 21} (2020), 100--117.

\bibitem{CCIs23} I.~Castro, I.~Castro-Infantes and J.~Castro-Infantes.
{\em Spherical Curves Whose Curvature Depends on Distance to a Great Circle.}
New Trends in Geometric Analysis. RSME Springer Series, vol 10. Springer, Cham. https://doi.org/10.1007/978-3-031-39916-9-2

\bibitem{CCIs24} I.~Castro, I.~Castro-Infantes and J.~Castro-Infantes.
{\em Helicoidal minimal surfaces in the 3-sphere: an approach via spherical curves.}
Rev.\ Real Acad.\ Cienc.\ Exactas Fis.\ Nat.\ Ser.\ A-Mat.\  {\bf 118}, 77 (2024). https://doi.org/10.1007/s13398-024-01574-3

\bibitem{Cha60} M.~Chasles.
{\em Resumé d'une théorie des coniques sphériques homofocales et des surfaces du second ordre homofocales.}
J.\ Math.\ Pures Appl.\ {\bf 5} (1860), 425--454.

\bibitem{Ch45} S.S~Chern.
{\em Some new characterizations of the Euclidean sphere.}
Duke Math.\ J.\ {\bf 12} (1945), 279--290.

\bibitem{dCD83} M.P.~do Carmo and M.~Dajczer.
{\em Rotation hypersurfaces in spaces of constant curvature}.
Trans.\ Amer.\ Math.\ Soc.\  {\bf 277} (1983), 685--709.
%
%\bibitem{E44} L.~Euler.
%{\em Methodus inveniendi lineas curvas: maximi minimive proprietate gaudentessive solutio problematis isoperimetrici latissimo sensu accepti (in Latin)}. 
%Opera Omnia: Series 1, Volume \textbf{24} (1744).

\bibitem{F93} R.~Ferr\'{e}ol.
{\em Encyclop\'{e}die des formes math\'{e}matiques remarquables.}
www.\ mathcurve.\ com

%\bibitem{H51} H.~Hopf.
%{\em \"{U}ber Fl\"{a}chen mit einer Relation zwischen den Hauptkr\"{ummungen}.}
%Math.\ Nachr.\ {\bf 4} (1951), 232--249.

%\bibitem{K15} W.~Kühnel.
%{\em Differential Geometry, Curves-Surfaces-Manifolds.}
%AMS, 2015.
%

%\bibitem{KS05} W.~Kühnel and M.~Steller.
%{\em On closed Weingarten surfaces.}
%Monatsh.\ Math.\ {\bf 146} (2005), 113--126.
%
%\bibitem{La69} H.B.~Lawson.
%{\em Local rigidity theorems for minimal hypersurfaces}.
%Ann.\ of Math.\  {\bf 89} (1969), 187--197.
%
%\bibitem{La70} H.B.~Lawson.
%{\em Complete minimal surfaces in $\s^3$}.
%Ann.\ of Math.\  {\bf 92} (1970), 335--374.

\bibitem{LYZ22} Y.R.~Luo, D.~Yang and X.Y.~Zhu
{\em Some characterizations of linear
	Weingarten surfaces in 3-dimensional space
	forms}.
Results Math.\  (2022) {\bf 77}:98. 
https://doi.org/10.1007/s00025-022-01633-4

\bibitem{Pr10} O.M.~Perdomo.
{\em Embedded constant mean curvature hypersurfaces on spheres}.
Asian J.\ Math.\  {\bf 14} (2010), 73--108.

\bibitem{Pr16} O.M.~Perdomo.
{\em Rotational surfaces in $\s^3$ with constant mean curvature}.
J.\ Geom.\ Anal.\  {\bf 26} (2016), 2155--2168.

%\bibitem{Pi85} U.~Pinkall.
%{\em Dupin hypersurfaces}.
%Math.\ Ann.\ {\bf 270}(3) (1985), 427--440.

\bibitem{P12} H. Pottmann, L. Shi and M. Skopenkov
{\em Darboux cyclides and webs from circles}. 
Comput.\ Aided Geom.\ Design, \textbf{29} (2012), 77--97.


\bibitem{Ri89} J.B.~Ripoll.
{\em Uniqueness of minimal rotational surfaces in $\s^3$}.
Amer.\ J.\ Math.\ {\bf 111} (1989), 537--547.

%\bibitem{Si07} 
%U.~Simon, 
%{\it Yau's problem on a characterization of rotational ellipsoids}, 
%Asian J.\ Math.\ {\bf 11} (2007), 361--372.

%\bibitem{S99} D.~Singer.
%{\em Curves whose curvature depends on distance from the origin.}
%Amer.\ Math.\ Monthly {\bf 106} (1999), 835--841.

%\bibitem{S08} D.~Singer.
%{\em Lectures on elastic curves and rods.}
%Curvature and Variational Modeling in Physics and Biophysics,
%AIP Conf.\ Proc.\ vol.\ {\bf 1002} (2008), 3--32.
%

\bibitem{T06} H.~Tranacher. 
{\em Sphärische Kegelschnitte didaktisch aufbereitet}. 
Diplomarbeit (master's thesis), Technische Universität Wien, 2006.

\bibitem{W61} J.~Weingarten. 
{\em Ueber eine Klasse auf einander abwickelbarer Fl\"{a}chen}. 
J.\ Reine Angew.\ Math.\, \textbf{59} (1861), 382--393.

\bibitem{Y82} 
S.T.\ Yau.
{\it Problem section, Seminar on Differential Geometry}.
Annals of Mathematical Studies \textbf{102}, Princeton University Press, 1982, pp. 669--706.

\bibitem{Z19} M.~Zhao, X.~Jia, C.~Tu, B.~Mourrain and W.~Wang.
{\em Enumerating the morphologies of non-degenerate Darboux cyclides}. 
Comput.\ Aided Geom.\ Design, \textbf{75} (2019), 101776.
https://doi.org/10.1016/j.cagd.2019.101776.

\end{thebibliography}
\end{document}